\documentclass[12pt,twoside,reqno]{amsart}

\usepackage{amssymb}
\usepackage{amsmath}
\usepackage{mathrsfs}
\usepackage{graphicx}
\usepackage{amsrefs}
\usepackage{cite}

\def\AA{{\mathcal{A}}}
\def\BB{{\mathcal{B}}}
\def\CC{{\mathcal{C}}}
\def\EE{{\mathcal{E}}}
\def\FF{{\mathcal{F}}}
\def\HH{{\mathcal{H}}}
\def\II{{\mathcal{I}}}
\def\JJ{{\mathcal{J}}}
\def\KK{{\mathcal{K}}}
\def\PP{{\mathcal{P}}}
\def\QQ{{\mathcal{Q}}}

\def\ZZ{{\mathcal{Z}}}

\def\R{{\mathbb R}}

\def\d{\delta}

\def\t{\theta}
\def\T{{\Theta}}
\def\e{\varepsilon}
\def\f{\varphi}
\def\s{\sigma}
\def\a{\alpha}
\def\g{\gamma}
\def\l{\lambda}

\def\k{\varkappa}

\DeclareMathOperator{\var}{Var}
\newtheorem{theorem}{Theorem}[section]
\newtheorem{lemma}[theorem]{Lemma}
\newtheorem{corollary}[theorem]{Corollary}
\newtheorem{proposition}[theorem]{Proposition}

\numberwithin{equation}{section}
\newcommand{\supp}{\mathop{\rm supp}\nolimits}

\begin{document}

\title[$L^2$-norm and estimates from below for Riesz transforms]
{$L^2$-norm and estimates from below\\
for Riesz transforms on Cantor sets}

\author{V.~Eiderman}
\address{Vladimir Eiderman, Department of  Mathematics, University of Kentucky, Lexington, KY}
\email{eiderman@ms.uky.edu}
\author{A.~Volberg}
\address{Alexander Volberg, Department of  Mathematics, Michigan State University, East Lansing, MI}
\email{volberg@math.msu.edu}
\thanks{Research of the second named author was supported in part by NSF grants  DMS-0501067}.
\begin{abstract}
The aim of this paper is to estimate the $L^2$-norms of vector-valued Riesz transforms $R_{\nu}^s$ and the norms of Riesz operators on Cantor sets in $\R^d$, as well as
to study the distribution of values of $R_{\nu}^s$. Namely, we show that this distribution is ``uniform'' in the following sense. The values of $|R_{\nu}^s|^2$
which are comparable with its average value are attended on a ``big'' portion of a Cantor set.
We apply these results to give examples demonstrating the sharpness of our previous estimates for the set of points where Riesz transform is large, and for the
corresponding Riesz capacities. The Cantor sets under consideration are different from the usual corner Cantor sets. They are constructed by means a certain
process of regularization introduced in the paper.
\end{abstract}

\maketitle

\section{Introduction}

Let $\s_0,\dots,\s_n$ be a finite sequence of positive numbers such that
\begin{equation}\label{f11}
2\s_{j+1}\le\s_j,\quad j=0,\dots,n-1.
\end{equation}
This sequence determines the corner Cantor set $E_n$ of generation $n$ in $\mathbb R^d$, such that the $j$-th generation consists of $2^{dj}$ cubes of edge length $\s_j$, each of these cubes contains $2^d$ corner cubes of the $(j+1)$-th generation, and so on. For brevity, we will call $E_n$ ``a Cantor set'' instead of ``a Cantor set of generation $n$''. There is a number of papers on estimates of various capacities, norms of integral transforms and operators, etc., on such Cantor sets. These estimates demonstrate the sharpness of various inequalities where the bounds are attained on Cantor sets; they are also of independent interest. But besides the necessary condition \eqref{f11}, there are certain additional conditions on $\s_j$ in many cases. In the present paper we associate with given numbers $\s_j$ satisfying {\it only} the condition \eqref{f11}, the ``regularized'' sequence $\{\ell_j\}_{j=1}^n$ such that $\s_j\approx\ell_j$, $j=1,\dots,n$, and construct the (non-corner) Cantor set $E_n$ formed by $2^{dn}$
cubes of edge length $\ell_n$. Since the corner and non-corner Cantor sets have similar structure, it is unimportant for applications which set to use.

For a nonnegative finite Borel measure $\nu$ in $\mathbb R^d$, $d\ge1$, and $s>0$, $\e>0$, define the $\e$-truncated $s$-Riesz transform of $\nu$ by
$$
R_{\nu,\e}^s(x)=\int_{|y-x|>\e}K^s(y-x)\,d\nu(y),
$$
where
$$
K^s(x)=\frac x{|x|^{s+1}},\quad x\in\mathbb R^d\setminus\{0\}.
$$
If the limit
$$
R_{\nu}^s(x):=\lim_{\e\to0+}R_{\nu,\e}^s(x)
$$
exists, we shall call it the $s$-Riesz transform of $\nu$ at $x$. To consider all finite Borel measures and all points $x\in\mathbb R^d$, one introduces the quantity
that always makes sense, namely the so called maximal $s$-Riesz transform
$$
R_{\nu,\ast}^s(x)=\sup_{\e>0}|R_{\nu,\e}^s(x)|
$$
(note that $R_{\nu,\e}^s(x)$ and $R_{\nu}^s(x)$ are vectors and $R_{\nu,\ast}^s(x)$ is a number).

Besides $R_{\nu,\e}^s$ and $R_{\nu,\ast}^s$, we need the $\e$-truncated $s$-Riesz operator defined by
$$
\mathfrak{R}_{\nu,\e}^s f(x)=\int_{|y-x|>\e}K^s(y-x)f(y)\,d\nu(y),\quad f\in L^2(\nu),\quad \e>0.
$$
For every $\e>0$, the operator $\mathfrak{R}_{\nu,\e}^s$ is bounded on $L^2(\nu)$. We set
$$
\pmb|\mathfrak{R}_{\nu}^s\pmb|:=\sup_{\e>0}\|\mathfrak{R}_{\nu,\e}^s\|_{L^2(\nu)\to L^2(\nu)}.
$$

Later on we denote by $c,C,c',\dots$ (without indices) positive constants which may vary from line to line.

Let $E_n$ be the corner Cantor set generated by a sequence $\s_0,\dots,\s_n$, and consisting of $2^{dn}$ cubes
$E_{n,k}$. Let $\mu$ the probability measure uniformly distributed on each cube $E_{n,k}$ with $\mu(E_{n,k})=2^{-dn}$.
Mateu and Tolsa \cite{MT} proved that if $0<s<d$, $(2+\d)\s_{j+1}\le\s_j,\ \d>0$, and $\t_{j+1}\le\t_j$ with
$$
\t_j=\frac{2^{-dj}}{\s_j^s},
$$
then
\begin{equation}\label{f12}
c\biggl[\sum_{j=1}^n\t_j^2\biggr]^{1/2}\le\pmb|\mathfrak{R}_{\mu}^s\pmb|\le C\biggl[\sum_{j=1}^n\t_j^2\biggr]^{1/2},
\end{equation}
where the constants $c,C$ depend only on $\d,\ d$ and $s$. In fact, Mateu and Tolsa proved a stronger assertion than the estimate from below: for sufficiently small $\e$,
\begin{equation}\label{f13}
\|R_{\mu,\e}^s\|_{L^2(\mu)}^2>c\sum_{j=1}^n\t_j^2.
\end{equation}
This result was refined by Tolsa in \cite{T1}, where the condition about monotonicity of densities $\t_j$ was dropped.
A more general class of Cantor sets for $s=d-1$ (again under the condition $\t_{j+1}\le\t_j$) was considered in \cite[Theorem~3.1]{GPT}.

The estimate from above in \eqref{f12} was also obtained in \cite{ENV} by another method.

The arguments in \cite{MT} and especially in \cite{T1} are rather complicated. We give two independent proofs of \eqref{f13} for our ``regularized'' Cantor set. The first (direct) proof is considerably simpler than in
\cite{MT}, \cite{T1}. The second approach gives the desired inequality as a corollary of the following more delicate result. We shall prove that the inequality
$|R_{\mu}^s(x)|^2>c\sum_{j=1}^n\t_j^2$ (and therefore the analogous estimate for $|R_{\mu,\e}^s|$) holds on a ``big'' portion of $E_n$.
We also consider the related problem in a more general setting and give certain applications. In particular, we establish the two-sided estimate of the Riesz capacity associated with $R_\nu^s$. This estimate is a refined version (for non-corner Cantor sets) of the corresponding results in \cite{MT}.

\medskip
We conclude this section with the construction of ``regularized'' non-corner Cantor sets.
Let $\s_0,\dots,\s_n$ be a finite sequence of positive numbers satisfying \eqref{f11}, and let two parameters $\a\in(0,\frac12)$ and $T\in(1,\frac1{2\a})$ be given. (Later on $\a, T$ will depend on $d$ and $s$.)
Define the set $J=\{j_1,\dots,j_m\}$, $m\le n$, of indices inductively in the following way: $j_1=1$; if $j_p\in J$, $1\le j_p<n$, and $\s_n\le\a2^{-(n-j_p)}\s_{j_p}$, then
$j_{p+1}$ is the least $j>j_p$, such that $\s_j\le\a2^{-(j-j_p)}\s_{j_p}$; if $j_p\in J$, $1\le j_p<n$, and $\s_n>\a2^{-(n-j_p)}\s_{j_p}$, then $j_{p+1}=n$. Thus,
\begin{gather*}
1=j_1<j_2<\cdots<j_m=n,\\
\a2^{-(j-j_p)}\s_{j_p}<\s_j\le2^{-(j-j_p)}\s_{j_p},\quad j_p\le j<j_{p+1},\quad p=1,\dots,m-1.
\end{gather*}
We set
\begin{gather*}
\ell_j=2^{-(j-j_p)}\s_{j_p},\quad j_p\le j<j_{p+1},\quad p=1,\dots,m-1;\\
\ell_n=\min\{\s_n,\,\a2^{-(n-j_{m-1})}\s_{j_{m-1}}\}.
\end{gather*}
Clearly (see \eqref{f11}),
\begin{equation}\label{f14}
\begin{array}[b]{c}
\s_j\le\ell_j<\a^{-1}\s_j,\ 1\le j<n;\quad \a\s_n\le\ell_n\le\s_n;\\ [.1in]
\ell_{j_{p+1}}\le\a2^{-(j_{p+1}-j_p)}\ell_{j_p},\quad p=1,\dots,m-1.
\end{array}
\end{equation}
Hence, $\a\s_j\le\ell_j<\a^{-1}\s_j,\ 1\le j\le n$.

For $w=(u_1,\dots,u_d)\in\R^d$ and $\ell>0$ let $Q(w,\ell)$ be the cube
\begin{equation}\label{f15}
Q(w,\ell)=\{x=(t_1,\dots,t_d)\in\R^d:\ |t_i-u_i|\le\tfrac12\ell,\ i=1,\dots,d\}.
\end{equation}
Construct the Cantor set $E_n$ recursively as follows. For $p=1$ we set $E_0=Q_{1,1}=Q(0,2T\ell_1)$. Take $2^d$ closed corner cubes $E_{1,k}$, $k=1,\dots,2^d$, of edge length
$\ell_1$ (i.e.~distinct cubes lying inside $Q_{1,1}$ with edges parallel to the edges of $Q_{1,1}$, such that each cube $E_{1,k}$ contains a vertex of $Q_{1,1}$).

\medskip

\includegraphics[height=110mm]{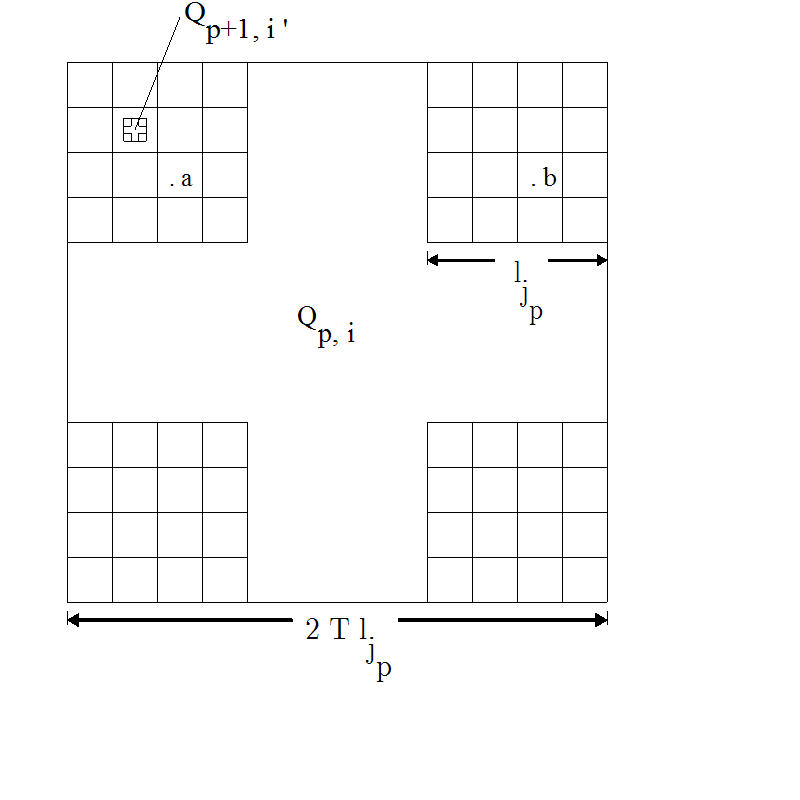}

Suppose that the cubes $E_{j_p,k}$, $k=1,\dots,2^{dj_p}$, of edge length $\ell_{j_p}$, $j_p<n$, are already defined. Partition each cube $E_{j_p,k}$ into $2^{d(j_{p+1}-j_p-1)}$
equal subcubes  $Q(w_{p+1,i'},2^{-(j_{p+1}-j_p-1)}\ell_{j_p})$. (In the figure above $d=2$, $j_{p+1}-j_p-1=2$.) If $j_{p+1}-j_p-1>0$, we may consider this partition as $j_{p+1}-j_p-1$ sequential partitions of $E_{j_p,k}$, such that on
$j$-th step, $j_p<j\le j_{p+1}-1$, we split each cube $E_{j-1,l}$ into $2^d$ cubes $E_{j,i'}$ of edge length $\ell_j$. Consider the cubes
$Q_{p+1,i'}=Q(w_{p+1,i'},2T\ell_{j_{p+1}})$. Remark that by \eqref{f14},
\begin{equation}\label{f16}
2T\ell_{j_{p+1}}\le2T\a\,2^{-(j_{p+1}-j_p)}\ell_{j_p}<2^{-(j_{p+1}-j_p)}\ell_{j_p}=\tfrac12 2^{-(j_{p+1}-j_p-1)}\ell_{j_p}.
\end{equation}
Take $2^d$ closed corner cubes $E_{j_{p+1},k}$ of edge length $\ell_{j_{p+1}}$ in each $Q_{p+1,i'}$. We get $2^{dj_{p+1}}$ cubes $E_{j_{p+1},k}$, and set
$$
Q_{p+1}=\bigcup_{i'=1}^{2^{d(j_{p+1}-1)}}Q_{p+1,i'};\quad E_{j_{p+1}}=\bigcup_{k=1}^{2^{dj_{p+1}}}E_{j_{p+1},k}.
$$
For $p+1=m$ we obtain the desired set $E_n$.

\section{Main results}

Our first theorem shows that under certain assumptions $|R_{\nu}^s|^2$ is comparable with its average value on a set of ``big'' measure, and this
property holds not only on Cantor sets. It means that the distribution of values of Riesz transform is uniform in a certain sense.

Set $B(x,r):=\{y\in\R^d:\ |y-x|<r\}$, and denote by $\Sigma_s$ the class of nonnegative Borel measures $\eta$ in $\R^d$ such that

\begin{equation}\label{f21}
\eta(B(x,r))\le r^s \quad\text{for all}\ x\in\R^d\text{ and }r>0.
\end{equation}

\begin{theorem}\label{th21} Suppose that $\eta\in\Sigma_s$, $\|\eta\|<\infty$, $R_{\eta}^s(x)$ exists $\eta$-a.e., $\pmb|\mathfrak{R}_{\eta}^s\pmb|\le1$, and
\begin{equation}\label{f22}
\|R_{\eta}^s\|_{L^2(\eta)}^2\ge a\|\eta\|,\quad a>0.
\end{equation}
Then for every $b\in(0,a)$ we have
\begin{equation}\label{f23}
\eta\{x:|R_{\eta}^s(x)|^2>b\}>c(a-b)^2\|\eta\|,\quad c=c(d,s).
\end{equation}

On the other hand, obviously \eqref{f23} implies \eqref{f22} with $bc(a-b)^2$ instead of $a$. The analogous statements hold for $R_{\eta,\ast}^s$.
\end{theorem}

We deduce Theorem \ref{th21} in Section 3 from a deep result by Nazarov, Treil and Volberg (Theorem 1.1 in \cite{NTV}). In Section 4 we obtain the following estimates for $\|R_{\mu}^s\|_{L^2(\mu)}^2$
and $\pmb|\mathfrak{R}_{\mu}^s\pmb|$. As before, we denote by $\mu$ the probability measure
uniformly distributed on each cube $E_{n,k}$ with $\mu(E_{n,k})=2^{-dn}$.

\begin{theorem}\label{th22} Let an integer $d\ge1$ and $s\in(0,d)$ be given. There are constants $\a\in(0,\frac12)$, $T\in(1,\frac1{2\a})$, depending only on $d$, $s$, and such that for any positive numbers $\s_1,\dots,\s_n$, satisfying \eqref{f11} with $1\le j\le n-1$, and for the corresponding Cantor set $E_n$,
\begin{align}
c\biggl[\sum_{j=1}^n\t_j^2\biggr]^{1/2}&\le\|R_{\mu}^s\|_{L^2(\mu)}\le C\biggl[\sum_{j=1}^n\t_j^2\biggr]^{1/2},\quad \t_j=\frac{2^{-dj}}{\ell_j^s},\label{f24}\\
c\biggl[\sum_{j=1}^n\t_j^2\biggr]^{1/2}&\le\pmb|\mathfrak{R}_{\mu}^s\pmb|\le C\biggl[\sum_{j=1}^n\t_j^2\biggr]^{1/2},\label{f25}
\end{align}
where the positive constants $c,C$ depend only on $d$ and $s$.
\end{theorem}
(We use the same notation $\t_j$ for values slightly different from the ones in  \eqref{f12}. Clearly, the corresponding relations in both cases are equivalent.)

Set $\eta:=c\,[\sum_{j=1}^n\t_j^2]^{-1/2}\mu$. If $c=c(d,s)>0$ is small enough then $\eta\in\Sigma_s$, and
$\pmb|\mathfrak{R}_{\eta}^s\pmb|\le1$ by the upper bound in \eqref{f25}. Moreover, the first inequality in \eqref{f24} implies \eqref{f22} with $a=a(d,s)$. Thus, Theorem \ref{th21} immediately yields the inequality
\begin{equation}\label{f26}
\mu\bigg\{x:|R_{\mu}^s(x)|^2> c\sum_{j=1}^n\t_j^2\bigg\}>\d_0,\quad \d_0=\d_0(d,s)>0.
\end{equation}
The existence of Cantor-type sets satisfying \eqref{f26} was established in \cite[Section 7]{ENV} using probabilistic arguments. But a concrete set was not presented. The particular case $d=2$, $s=1$, $\s_j=4^{-j}$, was considered in
\cite{AE}. A more general class of plane corner Cantor sets was treated in \cite{E1}.

Clearly, \eqref{f26} implies the estimates from below in \eqref{f24}, \eqref{f25} (in fact, the estimates from above were obtained in \cite{ENV} -- see Section 4 of the present paper for details). In Section 5 we give a completely different proof of \eqref{f26} without the use of Theorem 1.1 in \cite{NTV} and of Theorem~\ref{th22}. This independent approach allows us to consider a more general class of measures and wider range of $s$.

In Section 6 we consider the capacity $\g_{s,+}(E)$ of a compact set $E\subset\R^d$ defined by the equality
$$
\g_{s,+}(E):=\sup\{\|\nu\|:\nu\in M_+(E),\ \|R_{\nu}^s\|_{L^\infty(\R^d)}\le1\},
$$
where $M_+(E)$ is the class of positive Radon measures supported on $E$.
For $d=2,\ s=1$, $\g_{1,+}(E)$ is the analytic capacity $\g_+$, which is comparable with the analytic capacity $\g$ by the remarkable result of Tolsa \cite{To}. For $s=d-1,\ d\ge2$, $\g_{s,+}$ is comparable with the Lipschitz harmonic capacity (see \cite{V} and \cite[Section 10]{ENV} for details and references).

\begin{theorem}\label{th23} Let $d\ge1$, $s\in(0,d)$. For any finite sequence of positive numbers $\s_1,\dots,\s_n$, satisfying \eqref{f11} with $1\le j\le n-1$,
\begin{equation}\label{f27}
c\biggl[\sum_{j=1}^n\biggl(\frac{2^{-dj}}{\s_j^s}\biggr)^2\biggr]^{-1/2}\le\g_{s,+}(E_n)
\le C\biggl[\sum_{j=1}^n\biggl(\frac{2^{-dj}}{\s_j^s}\biggr)^2\biggr]^{-1/2},
\end{equation}
where the positive constants $c,C$ and the parameters $\a$, $T$ of the corresponding Cantor set $E_n$ depend only on $d$ and $s$.
\end{theorem}

For corner Cantor sets and $d=2,\ s=1$ (that is for analytic capacity $\g_+$), the estimates \eqref{f27} were
obtained in \cite{E} under the additional assumption $\s_{j+1}/\s_j\le\l<1/3$. A different proof has been given in
\cite{T}. The corresponding inequalities for $\g$ were proved in \cite{MTV} (before the Tolsa's result \cite{To} about comparability of $\g$ and $\g_+$).
The case of the Lipschitz harmonic capacity (i.~e. $d\ge2,\ s=d-1$) was treated in \cite{MT} under the assumptions
$(2+\d)\s_{j+1}\le\s_j,\ \d>0$, and $\t_{j+1}\le\t_j$. Thus, Theorem \ref{th23} is a refined version of these results for ``regularized" Cantor sets.
It is noted at the end of \cite{MT} that \eqref{f27} holds under the same assumptions for the signed Riesz capacity $\g_s$ as well. In fact, our proof of Theorem \ref{th23} is a modification of the arguments in \cite{MT} and \cite{MPV}, and these arguments also work for $\g_s$ and for ``regularized" Cantor sets under the additional assumption $(2+\d)\s_{j+1}\le\s_j$, but without monotonicity of $\t_{j}$. It is known that $\g_s\approx\g_{s,+}$ for $d=2$, $s=1$ \cite{To}, for $d\ge2$, $0<s<1$ \cite[Theorem 1.1]{MPV}, and for $d>2$, $s=d-1$ \cite{V}. As far as we know, the validity of this relation in other cases is an open problem. The extension to bilipschitz images of corner Cantor sets from \cite{MT} is given in \cite{GPT}.

In Section 7 we give ``the limit case'' of Theorem \ref{th23}, when the sequence $\{\s_j\}$ is infinite. We use the obtained estimates to demonstrate the sharpness of results in \cite{ENV}. In particular, we consider the problem of comparison of the capacity $\g_{s,+}$ and Hausdorff content.

\section{Proof of Theorem \ref{th21}}

\begin{lemma}\label{le31} If $f$ is a non-negative function, $\nu$ is an arbitrary probability measure, and
$$
\int f\,d\nu\ge L,\quad \int f^2\,d\nu\le AL^2,\quad L>0,
$$
then for $\d\in(0,1)$ we have
$$
\nu\{x:f\ge\d L\}\ge\frac{(1-\d)^2}{4A}\,.
$$
\end{lemma}

\begin{proof} Clearly,
$$
\int_{f>2AL/(1-\d)} f\,d\nu\le \int \frac{1-\d}{2AL}f^2\,d\nu\le \frac{(1-\d)L}2\,.
$$
Assume that
$$
\nu\{x:f\ge\d L\}<\frac{(1-\d)^2}{4A}\,.
$$
Then
$$
L-\frac{(1-\d)L}2\le\int_{f\le2AL/(1-\d)} f\,d\nu<\d L+\frac{2AL}{1-\d}\cdot\frac{(1-\d)^2}{4A}= \frac{(1+\d)}2L\,.
$$
Since the left hand side is equal to the same number $\frac{(1+\d)}2L$, we come to a contradiction.
\end{proof}

\begin{proof}[Proof of Theorem \ref{th21}] By \cite[Theorem 1.1, p.~467--468]{NTV}, the uniform boundedness of the  cut-off Calder\'on-Zygmund operators $T^{(\e)}$ on $L^2(\eta)$ implies the boundedness of $T$ and of the corresponding maximal singular operator on $L^p(\eta)$ for every $p\in(1,\infty)$. Applying this theorem for $f(x)\equiv1$, $p=4$, we get
\begin{align*}
\int|R_{\eta}^s(x)|^4\,d\eta(x)=&\|R_{\eta}^s\|_{L^4(\eta)}\le(A')^4\|\eta\|,\\
&\|R_{\eta,\ast}^s\|_{L^4(\eta)}\le(A'')^4\|\eta\|,
\end{align*}
where the constants $A', A''$ depend only on $d$ and $s$ (the last statement follows from the proof of Theorem~1.1 in \cite{NTV}).
Lemma \ref{le31} with $f(x)=|R_{\eta}^s(x)|^2$, $\nu=\eta/\|\eta\|$, $L=a$, $A=(A')^4a^{-2}$, $\d L=b$, yields \eqref{f23}, since $\frac{(1-\d)^2}{4A}=\frac{(a-b)^2}{4(A')^4}=c(a-b)^2$.

The proof of the corresponding statement for $R_{\eta,\ast}^s$ is essentially the same.
\end{proof}

\section{Proof of Theorem \ref{th22}}

We need some notation. Let $x_{j_p,k}$, $w_{p,i}$ be the centers of $E_{j_p,k}$ and $Q_{p,i}$ correspondingly, and let $E_{j_p}(x)$, $Q_{p}(x)$ be the cubes $E_{j_p,k}$, $Q_{p,i}$, containing $x$. Set
\begin{align*}
\xi_p(x)&=\int_{Q_{p}(x)\setminus Q_{p+1}(x)}\frac{y-x}{|y-x|^{s+1}}\,d\mu(y),\quad x\in Q_{p+1},\quad p=1,\dots,m-1;\\
\xi_m(x)&=\int_{Q_m(x)}\frac{y-x}{|y-x|^{s+1}}\,d\mu(y),\quad x\in E_n.
\end{align*}
Obviously,
$$
R_\mu^s(x)=\sum_{p=1}^m\xi_p(x).
$$

\begin{lemma}\label{le41} There exists $T_0=T_0(d,s)>1$, such that for any $\a\in(0,(2T_0)^{-1})$ and $T\in(T_0,(2\a)^{-1})$ we have
\begin{equation}\label{f41}
c\biggl(\frac{2^{-dj_p}}{T^s\ell_{j_p}^s}\biggr)^2\le\int_{E_n}\xi_p(x)^2\,d\mu(x)\le C\biggl(\frac{2^{-dj_p}}{\ell_{j_p}^s}\biggr)^2,\quad p=1,\dots,m,
\end{equation}
with $c$ and $C$ depending only on $d,s$.
\end{lemma}
\begin{proof}
We have
\begin{align*}
\xi_p(x)&=\int_{E_{j_p}(x)\setminus Q_{p+1}(x)}\frac{y-x}{|y-x|^{s+1}}\,d\mu(y)\\
&+\int_{Q_{p}(x)\setminus E_{j_p}(x)}\frac{y-x}{|y-x|^{s+1}}\,d\mu(y):=I_1(x)+I_2(x),\quad p=1,\dots,m-1;\\
\xi_m(x)&=\int_{E_{n}(x)}\frac{y-x}{|y-x|^{s+1}}\,d\mu(y)+ \int_{Q_m(x)\setminus E_{n}(x)}\frac{y-x}{|y-x|^{s+1}}\,d\mu(y):=I_1(x)+I_2(x).
\end{align*}
By \eqref{f16}, the cubes $Q_{p+1,i}$ in $E_{j_p}(x)$ are separated. Hence, up to a constant, $I_1(x)$ is majorized
by the integral over the measure uniformly distributed on $E_{j_p}(x)$ with density $2^{-dj_p}\ell_{j_p}^{-d}$, that is
\begin{equation}\label{f42}
|I_1(x)|\le C\,\frac{2^{-dj_p}}{\ell_{j_p}^d}\int_0^{\ell_{j_p}}\frac{t^{d-1}}{t^s}\,dt
=C'(d,s)\frac{2^{-dj_p}}{\ell_{j_p}^s}\, ,\quad x\in E_{j_p},\quad p=1,\dots,m.
\end{equation}
Suppose that $x\in E_{j_p,k}$ (i.~e. $E_{j_p}(x)=E_{j_p,k}$). We claim, that for sufficiently big $T$,
\begin{equation}\label{f43}
|I_2(x)-I_2(x_{j_p,k})|<\frac12\,|I_2(x_{j_p,k})|.
\end{equation}
Indeed,
\begin{align*}
|I_2(x)-I_2(x_{j_p,k})|&\le\int_{Q_{p}(x)\setminus E_{j_p}(x)}\bigg|
\frac{y-x}{|y-x|^{s+1}}-\frac{y-x_{j_p,k}}{|y-x_{j_p,k}|^{s+1}}\bigg|\,d\mu(y)\\
&<C(s)\int_{Q_{p}(x)\setminus E_{j_p}(x)}\frac{|x-x_{j_p,k}|}{|y-x_{j_p,k}|^{s+1}}\,d\mu(y)
<C'(d,s)\frac{2^{-dj_p}}{T^{s+1}\ell_{j_p}^s}\,.
\end{align*}
On the other hand, for $T$ big enough we have
\begin{equation}\label{f44}
c(d,s)\frac{2^{-dj_p}}{(T\ell_{j_p})^s}<|I_2(x)|<C(d,s)\frac{2^{-dj_p}}{(T\ell_{j_p})^s},\quad
x\in E_{j_p}.
\end{equation}
The lower bound in \eqref{f44} implies \eqref{f43}.

Obviously, $I_1(x')=-I_1(x'')$ whenever $x',x''\in E_{j_p,k}$ and $x',x''$ are symmetric with respect to $x_{j_p,k}$.
Hence, for ``half'' of the points $x\in E_{j_p,k}$, the angle between the vectors $I_2(x_{j_p,k})$ and $I_1(x)$ is less then or equal to $\pi/2$. For these $x$ we have
\begin{align*}
|\xi_p(x)|&=|I_1(x)+I_2(x)|\ge|I_1(x)+I_2(x_{j_p,k})|-|I_2(x)-I_2(x_{j_p,k})|\\
&\ge|I_2(x_{j_p,k})|-\frac12|I_2(x_{j_p,k})|>c'(d,s)\frac{2^{-dj_p}}{(T\ell_{j_p})^s}\,.
\end{align*}
We get the lower bound in \eqref{f41}. The upper bound follows directly from \eqref{f42} and \eqref{f44}.
\end{proof}

\begin{lemma}\label{le42} Let $\d=\min(\frac12(d-s),\frac12)$. Then
\begin{equation}\label{f45}
\bigg|\int_{E_n}\xi_p(x)\xi_q(x)\,d\mu(x)\bigg|<CT\biggl(\frac{\a}{2}\biggr)^{\d|p-q|}
\frac{2^{-dj_p}}{\ell_{j_p}^s}\cdot\frac{2^{-dj_q}}{\ell_{j_q}^s}\,,\quad p\ne q,
\end{equation}
where $C$ depends only on $d$ and $s$.
\end{lemma}
\begin{proof} By symmetry,
$$
\int_{Q_{p,i}}\xi_p(x)\,d\mu(x)=0,\quad p=1,\dots,m,\quad i=1,\dots,2^{d(j_p-1)}.
$$
Suppose that $p>q$. We have
\begin{equation}\label{f46}
\begin{split}
\int_{E_n}\xi_p(x)\xi_q(x)\,d\mu(x)&=\sum_{i=1}^{2^{d(j_p-1)}}\int_{Q_{p,i}}\xi_p(x)\xi_q(x)\,d\mu(x)\\
&=\sum_{i=1}^{2^{d(j_p-1)}}\int_{Q_{p,i}}\xi_p(x)[\xi_q(x)-\xi_q(w_{p,i})]\,d\mu(x).
\end{split}
\end{equation}
By \eqref{f42}, \eqref{f44},
\begin{equation}\label{f47}
|\xi_p(x)|<C\frac{2^{-dj_p}}{\ell_{j_p}^s}\,,\quad C=C(d,s),\quad x\in E_{j_p}.
\end{equation}
For $x\in Q_{p,i}$, we get
\begin{align*}
|\xi_q(x)-\xi_q(w_{p,i})|&\le\int_{Q_{q}(x)\setminus Q_{q+1}(x)}\bigg|
\frac{y-x}{|y-x|^{s+1}}-\frac{y-w_{p,i}}{|y-w_{p,i}|^{s+1}}\bigg|\,d\mu(y)\\
&<C\int_{Q_{q}(x)\setminus Q_{q+1}(x)}\frac{|x-w_{p,i}|}{|y-x|^{s+1}}\,d\mu(y)\\
&=C\int_{Q_{q}(x)\setminus Q_{q+1}(x)}\frac{|x-w_{p,i}|}{|y-x|^{1-\d}}\,\frac{d\mu(y)}{|y-x|^{s+\d}}
<C'\frac{T\ell_{j_p}}{\ell_{j_{q+1}}^{1-\d}}
\int_{Q_{q}(x)\setminus Q_{q+1}(x)}\frac{d\mu(y)}{|y-x|^{s+\d}}\,.
\end{align*}
As in Lemma \ref{le41}, we represent the last integral as the sum of the integrals over $E_{j_q}(x)\setminus Q_{q+1}(x)$ and $Q_{q}(x)\setminus E_{j_q}(x)$. The second integral is estimated exactly as in \eqref{f44}.
The first integral, as before, is majorized by the integral with uniformly distributed measure. Thus, the last bound does not exceed
\begin{multline*}
C\frac{T\ell_{j_p}}{\ell_{j_{q+1}}^{1-\d}}\bigg[\frac{2^{-dj_q}}{\ell_{j_q}^d}
\int_0^{\ell_{j_q}}\frac{t^{d-1}}{t^{s+\d}}\,dt+\frac{2^{-dj_q}}{(T\ell_{j_q})^{s+\d}}\bigg]
<C'\frac{T\ell_{j_p}}{\ell_{j_{q+1}}^{1-\d}}\cdot\frac{2^{-dj_q}}{\ell_{j_q}^{s+\d}}\\
=C'T\bigg(\frac{\ell_{j_p}}{\ell_{j_{q+1}}}\bigg)^{1-\d}\bigg(\frac{\ell_{j_p}}{\ell_{j_{q}}}\bigg)^{\d}
\frac{2^{-dj_q}}{\ell_{j_q}^{s}}
\le C'T\a^{\d(p-q)}2^{-\d(j_p-j_q)}\,\frac{2^{-dj_q}}{\ell_{j_q}^s},\quad C'=C'(d,s)
\end{multline*}
(in the last inequality we used the obvious relation $\ell_{j_p}\le\ell_{j_{q+1}}$ and \eqref{f14}).
Since $j_p-j_q\ge p-q$, we obtain the inequality
$$
|\xi_q(x)-\xi_q(w_{p,i})|<CT\a^{\d(p-q)}2^{-\d(p-q)}\frac{2^{-dj_q}}{\ell_{j_q}^s}\,.
$$
This inequality together with \eqref{f46}, \eqref{f47}, and the obvious relation
$\sum_{i=1}^{2^{d(j_p-1)}}\mu(Q_{p,i})=1$, imply \eqref{f45}.
\end{proof}

\begin{proof}[Proof of Theorem \ref{th22}] Set
\begin{equation}\label{f48}
\T_p=\frac{2^{-dj_p}}{\ell_{j_p}^s}\,.
\end{equation}
We start from the lower bound in \eqref{f24}. Obviously,
\begin{align*}
\|R_{\mu}^s\|_{L^2(\mu)}&=\int_{E_n}\bigg[\sum_{p=1}^{m}\xi_p(x)\bigg]^2\,d\mu(x)\\
&=\sum_{p=1}^{m}\int_{E_n}\xi_p(x)^2\,d\mu(x)+\sum_{p\ne q}\int_{E_n}\xi_p(x)\xi_q(x)\,d\mu(x)=:\Sigma_1+\Sigma_2.
\end{align*}
From \eqref{f41} we have
\begin{equation}\label{f49}
\Sigma_1\ge\frac{c}{T^{2s}}\sum_{p=1}^{m}\T_p^2.
\end{equation}
Enumerate $\T_{p}$ in decreasing order: $\T_{p_1}\ge\T_{p_2}\ge\dots\ge\T_{p_m}$. From \eqref{f45} we derive the estimate
\begin{equation}\label{f410}
\begin{split}
|\Sigma_2|&\le2CT\a^{\d}\bigg[\sum_{i=2}^m2^{-\d|p_1-p_i|}\T_{p_1}^2
+\sum_{i=3}^{m}2^{-\d|p_2-p_i|}\T_{p_2}^2+\dots+2^{-\d|p_{m-1}-p_m|}\T_{p_{m-1}}^2\bigg]\\
&<4CT\a^{\d}\bigg[\sum_{i=1}^{\infty}2^{-\d i}\bigg]\sum_{p=1}^m\T_{p}^2\,.
\end{split}
\end{equation}
We can choose $\a$ and $T$ in such a way that the constant $cT^{-2s}$ in \eqref{f49} is at least twice as big as
the constant before the last sum in \eqref{f410}. We have
$$
\|R_{\mu}^s\|_{L^2(\mu)}>c\sum_{p=1}^m\T_{p}^2\,,\quad c=c(d,s).
$$
To get the lower bound in \eqref{f24}, it remains to note that
\begin{equation}\label{f411}
\begin{split}
\sum_{j=j_p}^{j_{p+1}-1}\bigg(\frac{2^{-dj}}{\ell_{j}^{s}}\bigg)^2&=
\sum_{j=j_p}^{j_{p+1}-1}\bigg(\frac{2^{-dj}}{2^{-s(j-j_p)}\ell_{j_p}^{s}}\bigg)^2\\
&=\bigg(\frac{2^{-dj_p}}{\ell_{j_p}^{s}}\bigg)^2\cdot\sum_{j=j_p}^{j_{p+1}-1}2^{-2(d-s)(j-j_p)}
<C(d,s)\T_{p}^2\,,\quad p=1,\dots,m-1.
\end{split}
\end{equation}

Obviously, the estimate from below obtained in \eqref{f24} implies the lower bound in \eqref{f25}. To complete the proof of Theorem \ref{th22}, it is enough to get the upper bound in \eqref{f25}. But the proof of this estimate is literally the same as the proof of the corresponding estimate for corner Cantor sets in \cite[Corollary~3.5]{ENV}.
\end{proof}

\section{Estimates for the size of the set where $|R_\nu^s|$ is large.\\
Another approach}

In this section we develop an independent approach to obtaining the estimate \eqref{f26}, as well as
its generalizations and related results.

Let a finite sequence $\s_1,\dots,\s_n$, $2\s_{j+1}\le\s_j$, $j=1,\dots,n-1$, and constants $\a\in(0,\frac12)$,
$T\in(1,\frac1{2\a})$, $c_0>0$, be given. For the corresponding Cantor set $E_n$ and for a positive measure $\nu$ supported on $E_n$, set
$$
\tilde R_\nu^s(x)=\int_{E_n\setminus E_n(x)}\frac{y-x}{|y-x|^{s+1}}\,d\nu(x),
$$
where $E_n(x)$ is the cube $E_{n,k}$ containing $x$. Define the sets $\EE$, $\tilde\EE$ of cubes by the relations
\begin{align}
\EE&=\bigg\{E_{n,k}:|R_\nu^s(x_{n,k})|>\frac{c_0}{T^s}\biggl[\sum_{j=1}^n\t_j^2\biggr]^{1/2}\bigg\},\quad \t_j=\frac{2^{-dj}}{\ell_j^s},\label{f51}\\
\tilde\EE&=\bigg\{E_{n,k}:|\tilde R_\nu^s(x_{n,k})|>\frac{c_0}{T^s}\biggl[\sum_{j=1}^n\t_j^2\biggr]^{1/2}\bigg\}\label{f52}.
\end{align}
Here as before, $x_{n,k}$ is the center of $E_{n,k}$; in \eqref{f51} we assume that the values
$R_\nu^s(x_{n,k})$ exist.

\begin{theorem}\label{th51} For every $s\in(0,\infty)$ and every integer $d\ge1$, there exist constants $c_0$,
$T_0$, depending only on $d,s$, with the following properties.
Fix some $T>T_0$, $\a<\a_0(d,s,T)$, and a sequence $\s_1,\dots,\s_n$, $2\s_{j+1}\le\s_j$, $j=1,\dots,n-1$. Let $\nu$ be a measure supported on the corresponding Cantor set $E_n$, equally (but not necessarily uniformly) distributed on
each cube $E_{n,k}$, $k=1,\dots,2^{dn}$, with $\nu(E_{n,k})=2^{-dn}$. Then the number of cubes $E_{n,k}$ in
$\tilde\EE$ is comparable with the number of all cubes in $E_n$, that is,
\begin{equation}\label{f53}
\#\tilde\EE>c_12^{dn},
\end{equation}
where $c_1$ is an absolute constant. Moreover, if values $R_\nu^s(x_{n,k})$ exist, then
\begin{equation}\label{f54}
\#\EE>c_12^{dn}.
\end{equation}
The same conclusion holds if we replace $x_{n,k}$ in \eqref{f51}, \eqref{f52} by any fixed points
$x'_{n,k}\in E_{n,k}$, such that $x'_{n,k}-x_{n,k}=x'_{n,l}-x_{n,l}$, $k,l=1,\dots,2^{dn}$.
\end{theorem}

This theorem implies a useful corollary which will be given after the proof.

The proofs of \eqref{f53}, \eqref{f54}, and of the statement for points $x'_{n,k}$, are the same. For definiteness,
we consider the case \eqref{f54}.

\begin{lemma}\label{le52} For every finite sequence $\{a_p\}_{p=1}^m$ of positive numbers and for given $\d>0$,
there is a subsequence $\{a_{p_l}\}_{l=1}^K$, such that
\begin{gather}
2^{-\d(p_l-p)}a_p\le a_{p_l}\ \text{ for all } p\le p_l,\quad l=1,\dots,K;\label{f55}\\
\sum_{l=1}^K a_{p_l}^{\k}\ge c\sum_{p=1}^{m}a_p^\k,\quad \k>0,\label{f56}
\end{gather}
where $c$ depends only on $\d$ and $\k$.
\end{lemma}

\begin{proof} Let the maximal index $p_K$ be the least integer for which
$$
\max_{1\le p\le m}2^{-\d(m-p)}a_p=2^{-\d(m-p_K)}a_{p_K}
$$
(the value of $K$ will be determined by the construction in the sequel).

Suppose that indices $p_K, p_{K-1},\dots,p_{l+1}$ are already defined, and $p_K>p_{K-1}>\dots>p_{l+1}>1$. Then
$p_l$ is the least integer satisfying the relation
$$
\max_{1\le p< p_{l+1}}2^{-\d(p_{l+1}-p)}a_p=2^{-\d(p_{l+1}-p_l)}a_{p_l}.
$$
This equality implies that
\begin{equation}\label{f57}
2^{-\d(p_{l}-p)}a_p\le a_{p_l},\quad 1\le p<p_{l+1}.
\end{equation}
In particular, we get \eqref{f55}. Clearly, $p_1=1$. Using \eqref{f57} for $p\in[p_l,p_{l+1})$ with $p_{K+1}:=m+1$,
we get the following estimate:
$$
\sum_{p=1}^{m}a_p^\k=\sum_{l=1}^K\sum_{p=p_l}^{p_{l+1}-1}a_p^\k\le
\sum_{l=1}^K\sum_{p=p_l}^{p_{l+1}-1}2^{-\k\d(p-p_l)}a_{p_l}^{\k}
<\sum_{i=0}^\infty2^{-\k\d i}\sum_{l=1}^K a_{p_l}^{\k},
$$
which yields \eqref{f56}.
\end{proof}

We use Lemma \ref{le52} with $a_p=\T_p$, $\d=1/2$, $\k=2$, in order to extract the future subsequence from
$\{\T_p\}_{p=1}^m$ (the numbers $\T_p$ are defined by \eqref{f48}).

By Lemma \ref{le52}, there exists the set $\PP$ of indices such that
\begin{gather}
2^{-0.5(p-q)}\T_q\le \T_{p}\ \text{ for all } q\le p,\quad p\in\PP;\label{f58}\\
\sum_{p\in\PP}\T_{p}^2\ge c\sum_{q=1}^{m}\T_q^2,\label{f59}
\end{gather}
where $c$ is an absolute constant. Set $K=\#\PP$.

Each cube $E_{n,k}$, $k=1,\dots,2^{dn}$, can be represented in the form
$$
E_{n,k}=\KK_u\times\II_v,\quad u=1,\dots,2^{(d-1)n},\quad v=1,\dots,2^n,
$$
where $\KK_u$ is the projection of $E_{n,k}$ onto the hyperplane $t_1=0$, and $\II_v$ is the projection of $E_{n,k}$ onto the $t_1$-axis. For $u$ fixed, we associate the vector $\mathbf e^{(v)}=(e_1^{(v)},\dots,e_n^{(v)})$ with each
cube $E_{n,k}=\KK_u\times\II_v$ in the following way. The choice of an interval $\II_v$ can be viewed as a result
of $n$ subsequent choices (steps): starting from the interval $[-T\ell_1,T\ell_1]$, at $j$-th step, $j\ge1$, we
choose the left or the right of two equal subintervals of length $\ell_j$ in the preceding interval. We set
$e_j^{(v)}=-1$, if we choose the left subinterval at $j$-th step, and $e_j^{(v)}=1$ in the opposite case. Thus, we
get the one-to-one correspondence between the cubes $E_{n,k}=\KK_u\times\II_v$ with fixed $u$, and vectors
$\mathbf e^{(v)}$.

\begin{lemma}\label{le53} Let $T>T_0$, $\a<\a_0(d,s,T)$, and let the cubes $E_{n,k_i}=\KK_u\times\II_{v_i}$, $i=1,2$,
be such that
\begin{align*}
e_{j_p}^{(v_1)}&=-1,\quad e_{j_p}^{(v_2)}=1\quad \text{for some }\ p\in\PP,\\
e_j^{(v_1)}&=e_j^{(v_2)},\quad j\ne j_p.
\end{align*}
Suppose that a measure $\nu$ satisfies the conditions of Theorem \ref{th51}. Then
\begin{equation}\label{f510}
R_{1\nu}^s(x_{n,k_1})-R_{1\nu}^s(x_{n,k_2})>\frac{c_2}{T^s}\,\frac{2^{-dj_p}}{\ell_{j_p}^s}\,,
\end{equation}
where $R_{1\nu}^s(x)$ is the first component of a vector $R_{\nu}^s(x)$, and the constants $T_0$, $c_2$
depend only on $d,s$.
\end{lemma}

\begin{proof} In order to simplify notation, we set $x_{n,k_1}=a$, $x_{n,k_2}=b$ (see the figure in Section~1, where a possible location of $a$, $b$ is indicated). As before, we denote by $E(x)$,
$Q(x)$ the corresponding cubes containing $x$. Obviously,
$$
\int_{E_{j_p}(a)}\frac{y-a}{|y-a|^{s+1}}\,d\nu(y)=\int_{E_{j_p}(b)}\frac{y-b}{|y-b|^{s+1}}\,d\nu(y)\,,\quad
b-a=(2T-1)\ell_{j_p}.
$$
Hence,
\begin{align*}
&R_{\nu}^s(a)-R_{\nu}^s(b)=\int_{E_n}\bigg(\frac{y-a}{|y-a|^{s+1}}-\frac{y-b}{|y-b|^{s+1}}\bigg)\,d\nu(y)\\
&=\int_{Q_p(a)\setminus E_{j_p}(a)}\frac{y-a}{|y-a|^{s+1}}\,d\nu(y)-
\int_{Q_p(a)\setminus E_{j_p}(b)}\frac{y-b}{|y-b|^{s+1}}\,d\nu(y)\\
&+\sum_{q=2}^p \int_{E_{j_{q-1}}(a)\setminus Q_q(a)}\bigg(\frac{y-a}{|y-a|^{s+1}}-\frac{y-b}{|y-b|^{s+1}}\bigg)\,d\nu(y)\\
&+\sum_{q=2}^p \int_{Q_{q-1}(a)\setminus E_{j_{q-1}}(a)}\bigg(\frac{y-a}{|y-a|^{s+1}}-\frac{y-b}{|y-b|^{s+1}}\bigg)\,d\nu(y)
=: I_1-I_2+\Sigma_1+\Sigma_2
\end{align*}
(for $p=1$ the last sums are absent). If $T$ is greater than certain $T_0=T_0(d,s)$ then
\begin{equation}\label{f511}
\{\text{the first component of the vector }( I_1-I_2)\}\ge\frac{2c_2}{T^s}\,\frac{2^{-dj_p}}{\ell_{j_p}^s}=\frac{2c_2}{T^s}\T_p\,,
\quad c_2=c_2(d,s).
\end{equation}

We claim that
\begin{equation}\label{f512}
\ell_{j_p}\sum_{q=2}^p\frac{2^{-dj_{q-1}}}{\ell_{j_{q-1}}^{s+1}}(j_q-j_{q-1})<C\a\T_p\,,\quad p\in\PP,
\quad 0<\a<\frac12\,,
\end{equation}
where $C$ is an absolute constant. Indeed, by \eqref{f58} the the left hand side of \eqref{f512} does not exceed
$$
\ell_{j_p}\T_p\sum_{q=2}^p\frac{2^{0.5(p-q)}}{\ell_{j_{q-1}}}(j_q-j_{q-1}).
$$
The inequality \eqref{f14} yields the estimate
$$
\ell_{j_{q-1}}\ge\a^{-(p-q+1)}2^{j_p-j_{q-1}}\ell_{j_p}\,.
$$
Continuing the estimation, we get \eqref{f512}:
$$
\T_p\sum_{q=2}^p2^{0.5(p-q)}\a^{p-q+1}2^{-(j_p-j_{q-1})}(j_q-j_{q-1})<C\a\T_p\,.
$$

Since $|y-a|<2|y-b|$ in the integrals from $\Sigma_1$ and $\Sigma_2$, we have
\begin{equation}\label{f513}
\bigg|\frac{y-a}{|y-a|^{s+1}}-\frac{y-b}{|y-b|^{s+1}}\bigg|\le C(s)\frac{|b-a|}{|y-a|^{s+1}}
<2C(s)\frac{T\ell_{j_p}}{|y-a|^{s+1}}\,.
\end{equation}
Thus,
$$
|\Sigma_1|<CT\ell_{j_p}\sum_{q=2}^p \int_{E_{j_{q-1}}(a)\setminus Q_q(a)}\frac{1}{|y-a|^{s+1}}\,d\nu(y),\quad C=C(s).
$$
As in the proof of Lemma \ref{le41}, we majorize the last integrals by the integrals over the measures uniformly
distributed on $E_{j_{q-1}}(a)$ with densities $2^{-dj_{q-1}}/\ell_{j_{q-1}}^d$. We get
\begin{equation}\label{f514}
|\Sigma_1|<C(d,s)T\ell_{j_p}\sum_{q=2}^p \frac{2^{-dj_{q-1}}}{\ell_{j_{q-1}}^d}
\int_{2^{-(j_q-j_{q-1}-1)}\ell_{j_{q-1}}}^{\ell_{j_{q-1}}}\frac{t^{d-1}}{t^{s+1}}\,dt.
\end{equation}

We consider three cases.

{\bf Case 1.} $s<d-1$. Then
$$
|\Sigma_1|<C(d,s)T\ell_{j_p}\sum_{q=2}^p \frac{2^{-dj_{q-1}}}{\ell_{j_{q-1}}^d}\,\ell_{j_{q-1}}^{d-s-1}
\stackrel{\eqref{f512}}{<}C'(d,s)T\a\T_p<\frac{c_2}{2T^s}\T_p
$$
for $\a<c_2[2C'(d,s)T^{s+1}]^{-1}$, where $c_2$ is the constant in \eqref{f511}.

{\bf Case 2.} $s=d-1$. By \eqref{f514} and \eqref{f512} we have
$$
|\Sigma_1|<CT\ell_{j_p}\sum_{q=2}^p \frac{2^{-dj_{q-1}}}{\ell_{j_{q-1}}^{s+1}}\,
(j_q-j_{q-1}-1)\stackrel{\eqref{f512}}{<}C''T\a\T_p<\frac{c_2}{2T^s}\T_p,\quad C''=C''(d,s),
$$
if $\a<c_2[2C''T^{s+1}]^{-1}$.

{\bf Case 3.} $s>d-1$. Again by \eqref{f514},
\begin{align*}
|\Sigma_1|&<CT\ell_{j_p}\sum_{q=2}^p \frac{2^{-dj_{q-1}}}{\ell_{j_{q-1}}^d}
2^{(s-d+1)(j_q-j_{q-1}-1)}\frac1{\ell_{j_{q-1}}^{s-d+1}}\\
&\stackrel{\eqref{f14}}{<}CT\ell_{j_p}\sum_{q=2}^p \frac{2^{-dj_p}2^{d(j_p-j_{q-1})}\a^{(s+1)(p-q+1)}}
{\ell_{j_p}^{s+1}2^{(s+1)(j_p-j_{q-1})}}2^{(s-d+1)(j_q-j_{q-1}-1)}\\
&<C'T\a^{s+1}\T_p\sum_{q=2}^p 2^{(s-d+1)(j_q-j_p)}<\frac{c_2}{2T^s}\T_p, \quad \a^{s+1}<\frac{c_2}{C''T^{s+1}}.
\end{align*}

As concerns $\Sigma_2$, we have
$$
|\Sigma_2|<CT\ell_{j_p}\sum_{q=2}^p \frac{2^{-dj_{q-1}}}{(T\ell_{j_{q-1}})^{s+1}}
\stackrel{\eqref{f512}}{<}\frac{c_2}{2T^s}\T_p, \quad \a<c(d,s).
$$
The estimates for $\Sigma_1$, $\Sigma_2$ and \eqref{f511} yield \eqref{f510}.
\end{proof}

The next lemma is a particular case of Lemma 7.3 in \cite{ENV}.

\begin{lemma}\label{le54} Let $\l_i,\ i=1,\dots,K$, be positive numbers. Let $\zeta_i,\ i=1,\dots,K$, be independent
random variables satisfying
\begin{equation}\label{f515}
|\zeta_i|\le C\l_i,\quad \sum_{i=1}^K\var(\zeta_i)\ge c\sum_{i=1}^K\l_i^2\,.
\end{equation}
Then there exists $\beta=\beta(C,c)>0$ such that
\begin{equation}\label{f516}
\mathsf{P}\bigg\{\bigg|\sum_{i=1}^K\zeta_i\bigg|\ge\beta\bigg(\sum_{i=1}^K\l_i^2\bigg)^{1/2}\bigg\}\ge\beta.
\end{equation}
\end{lemma}

We need the following Lemma 10.2 in \cite{E1}.

\begin{lemma}\label{le55} Let positive numbers $\l_i,\ i=1,\dots,K$, be given. For each vector $\mathbf f^{(r)}
=(f^{(r)}_1,\dots,f^{(r)}_K)$, $r=1,\dots,2^K$, with components equal to plus or minus ones we set
$$
S(\mathbf f^{(r)})=\sum_{i=1}^K f^{(r)}_i\l_i.
$$
Let $\AA$ be a set of different vectors $\mathbf f^{(r)}$ such that $\AA\subseteq\{\mathbf f^{(r)}:
S(\mathbf f^{(r)})\le0\}$. Then there exists a set $\BB\subseteq\{\mathbf f^{(r)}:
S(\mathbf f^{(r)})\ge0\}$ with the following properties:

1) all the vectors in $\BB$ are distinct;

2) $\#\AA\le\#\BB$;

3) each vector $\mathbf f'\in\BB$ can be obtained from some vector $\mathbf f\in\AA$ by replacing
some negative components of $\mathbf f$ (depending on $\mathbf f'$)
by positive components while keeping all the positive components of $\mathbf f$. (Different vectors $\mathbf f'$
can be obtained from the same $\mathbf f$.)
\end{lemma}

\begin{proof}[Proof of Theorem \ref{th51}] Set $\l_i=\T_{p_i}$, $p_i\in\PP$, $p_1<p_2<\dots<p_K$,
$\s^2=\sum_{i=1}^K \l_i^2$. Fix an integer $u$ and the components $e_j^{(v)},\ j\notin\PP$ of vectors
$\mathbf e^{(v)}=(e_1^{(v)},\dots,e_n^{(v)})$. Let $\FF$ be the set of cubes $E_{n,k}=\KK_u\times\II_v$ with fixed parameters indicated above. Clearly, $\FF$ consists of $2^K$ elements depending only on the choice of the components $e_j^{(v)}$ with $j\in\PP$. We can consider this sampling as a choice of a vector $\mathbf f^{(r)}$. We introduce the following subsets of $\FF$ (or equivalently, the subsets of vectors $\mathbf f^{(r)}$):
\begin{align*}
\QQ_1&=\{E_{n,k}\in\FF: |R_{1\nu}^s(x_{n,k})|>\frac{c_2}{4T^s}\beta\s\},\\
\QQ_2&=\{E_{n,k}\in\FF: |R_{1\nu}^s(x_{n,k})|\le\frac{c_2}{4T^s}\beta\s\},\\
\QQ_3&=\{E_{n,k}\in\FF: S(\mathbf f^{(r)})\le-\beta\s\},\\
\QQ_4&=\{E_{n,k}\in\FF: S(\mathbf f^{(r)})>-\beta\s\},
\end{align*}
where $\beta$ is the constant in Lemma \ref{le54} with $c=C=1$, and $c_2$ is the constant in \eqref{f510}.

We shall consider the choice of a vector $\mathbf f^{(r)}$, $r=1,\dots,2^K$, as a random event with the same probability $2^{-K}$ for all vectors. This defines a random variable $\Xi$ with values $S(\mathbf f^{(r)})$.
Obviously, we can interpret $\Xi$ as the sum of independent random variables $\zeta_i$, $i=1,\dots,K$, taking the
values $\l_i$ and $-\l_i$ with probability 1/2.

We now use Lemma \ref{le54}. In our case, $\var(\zeta_i)=\l_i^2$, $c=C=1$. Since distribution of $\Xi$ is
symmetric, \eqref{f516} yields the inequality
$$
\mathsf{P}\{\Xi\le-\beta\s\}\ge\frac{\beta}2\,.
$$
Hence,
\begin{equation}\label{f517}
\#\QQ_3\ge\beta2^{K-1}.
\end{equation}
Consider now the set $\AA:=\QQ_2\cap\QQ_3$. By $\AA$ we denote also the set of the corresponding vectors
$\mathbf f^{(r)}$. If $\AA=\varnothing$, then $\#\AA=0$, and we immediately arrive at the estimates
\eqref{f518} below.
Assume that $\AA\ne\varnothing$, and let $\BB$ be the set of vectors corresponding to the set $\AA$ by Lemma \ref{le55}. Consider an arbitrary vector $\mathbf f'\in\BB$. It can be obtained from some $\mathbf f\in\AA$
by the replacement of some negative components of $\mathbf f$ by positive ones. Let $\JJ$ be the set of indices of
the replaced components of $\mathbf f$. Let $\mathbf f=\mathbf f^{(r)}$, $\mathbf f'=\mathbf f^{(t)}$, and let
$E_{n,k}$, $E_{n,m}$ be the cubes associated with $\mathbf f^{(r)}$ and $\mathbf f^{(t)}$ correspondingly. Since
$S(\mathbf f^{(t)})\ge0$, it follows that
$$
\beta\s\le S(\mathbf f^{(t)})-S(\mathbf f^{(r)})=2\sum_{i\in\JJ}\l_i\,.
$$
On the other hand, applying Lemma \ref{le53} $\#\JJ$ times (successively for each component in $\JJ$) we get
$$
R_{1\nu}^s(x_{n,k})-R_{1\nu}^s(x_{n,m})>\frac{c_2}{T^s}\sum_{i\in\JJ}\l_i\ge\frac{c_2}{2T^s}\beta\s,
$$
so that
$$
|R_{1\nu}^s(x_{n,m})|\ge|R_{1\nu}^s(x_{n,k})-R_{1\nu}^s(x_{n,m})|-|R_{1\nu}^s(x_{n,k})|>\frac{c_2}{4T^s}\beta\s.
$$
Hence $\BB\subset\QQ_1$. Moreover, $\BB\subset\QQ_4$, since $S(\mathbf f^{(r)})\ge0$ for $\mathbf f^{(r)}\in\BB$.
Bearing in mind that $\#\AA\le\#\BB$, we obtain
$$
\#(\QQ_2\cap\QQ_3)=\#\AA\le\#\BB\le\#(\QQ_1\cap\QQ_4).
$$
Thus, for every set $\FF$ (that is, for each $u$ and for each collection of components $e_j^{(v)}$, $j\notin\PP$),
we have
\begin{equation}\label{f518}
\begin{split}
\#\QQ_1&=\#(\QQ_1\cap\QQ_3)+\#(\QQ_1\cap\QQ_4)\\
&\ge\#(\QQ_1\cap\QQ_3)+\#(\QQ_2\cap\QQ_3)=\#\QQ_3\stackrel{\eqref{f517}}{\ge}\beta2^{K-1}.
\end{split}
\end{equation}
Since there exist $2^{(d-1)n}$ values of $u$, and $2^{n-K}$ collection of components $e_j^{(v)}$ with $j\notin\PP$,
it follows that
$$
\#\{E_{n,k}: |R_{1\nu}^s(x_{n,k})|>\frac{c_2}{4T^s}\beta\s\}>\frac{\beta}2 2^{dn}.
$$
It remains to note that
$$
\s^2=\sum_{p\in\PP}\T_p^2\stackrel{\eqref{f59}}{\ge}c\sum_{q=1}^m\T_q^2\stackrel{\eqref{f411}}{>}
c'\sum_{j=1}^n\t_j^2.
$$
Theorem \ref{th51} is proved.
\end{proof}

Using Theorem \ref{th51} we show that for the class of measures $\nu$ satisfying the assumptions of this theorem,
the Riesz transform $R_\nu^s(x)$ is large on a ``big'' portion of $E_n$.

\begin{corollary}\label{co56} Let an integer $d\ge1$ and $s>0$ be given. There are constants $\a\in(0,\frac12)$, $T\in(1,\frac1{2\a})$, $c_3>0$, depending only on $d$, $s$, and such that for any positive numbers $\s_1,\dots,\s_n$, $2\s_{j+1}\le\s_j$, $1\le j\le n-1$, and for the corresponding Cantor set $E_n$,
\begin{equation}\label{f519}
\HH^d\bigg\{x\in E_n:|R_{\nu}^s(x)|> c_3\bigg[\sum_{j=1}^n\t_j^2\bigg]^{1/2}\bigg\}>\frac{c_1}2\HH^d(E_n)= \frac{c_1}22^{dn}\ell_n^d\,,\quad \t_j=\frac{2^{-dj}}{\ell_j^s}\,.
\end{equation}
Here $\HH^d$ is $d$-dimensional Hausdorff measure, and $\nu$ is a measure satisfying the conditions of
Theorem~\ref{th51}, such that $R_\nu^s(x)$ exists (in the sense of principal values) $\HH^d$-a.~e. on $E_n$.
\end{corollary}

\noindent{\bf Remark.} For $s\le d$ the condition about existence of $R_\nu^s(x)$ holds for any Borel measure
on $E_n$. Generally speaking, it is not correct for $s>d$.

\begin{proof} Set $T=T_0+1$, $c_3=c_0(2T^s)^{-1}$, where $c_0$ and $T_0$ are the constants from Theorem~\ref{th51}.
Choose a sufficiently small $\a$ which will be specified later, and consider the corresponding Cantor set $E_n$.
Let $\nu_0$ be the point measure with charges equal to $2^{-dn}$ and located at the centers $x_{n,k}$ of the cubes
$E_{n,k}$. We claim that
\begin{equation}\label{f520}
|\tilde R_\nu^s(x)-\tilde R_{\nu_0}^s(x)|<\frac{c_0}{4T^s}\biggl[\sum_{j=1}^n\t_j^2\biggr]^{1/2},\quad x\in E_n,
\end{equation}
if $\a$ is small enough. Indeed, let $x',x''$ be any two points lying in the same cube $E_{n,k}$
(that is, $E_n(x')=E_n(x'')=E_{n,k}$). In the same way as \eqref{f513} we obtain the inequality
\begin{equation}\label{f521}
\bigg|\frac{y-x'}{|y-x'|^{s+1}}-\frac{y-x''}{|y-x''|^{s+1}}\bigg|\le \frac{C\ell_n}{|y-x'|^{s+1}}\,,\quad C=C(s).
\end{equation}
Let $j\in[1,n-1]$, and let $p\in[1,m-1]$ be such that $j_p\le j<j_{p+1}$. By \eqref{f14} we get
\begin{equation}\label{f522}
\begin{split}
\ell_n&\le\a^{m-p}2^{-(j_m-j_p)}\ell_{j_p}=\a^{m-p}2^{-(j_m-j_p)}2^{j-j_p}\ell_j\\
&=\a^{m-p}2^{-(j_m-j)}\ell_j\le\a2^{-(n-j)}\ell_j
\end{split}
\end{equation}
(we recall that $j_m=n$). We use \eqref{f521} with $x\in E_{n,k}$ instead of $y$, and with $x'=y\in E_{n,i}$,
$x''=x_{n,i}$, $i\ne k$. Then
\begin{equation}\label{f523}
\begin{split}
&|\tilde R_\nu^s(x)-\tilde R_{\nu_0}^s(x)|\le\sum_{i\ne k}\int_{E_{n,i}}
\bigg|\frac{y-x}{|y-x|^{s+1}}-\frac{x_{n,i}-x}{|x_{n,i}-x|^{s+1}}\bigg|\,d\nu(y)\\
&\le C\ell_n\int_{E_{n}\setminus E_n(x)}\frac{d\nu(y)}{|y-x|^{s+1}}
\le C'\ell_n\sum_{j=1}^{n-1}\int_{\ell_j<|y-x|\le\ell_{j-1}}\frac{d\nu(y)}{|y-x|^{s+1}}\\
&\le C''\ell_n\sum_{j=1}^{n-1}\frac{2^{-dj}}{\ell_j^{s+1}}
\le C''\max_{1\le j\le n-1}\bigg(\frac{2^{-dj}}{\ell_j^s}\bigg)\cdot\sum_{j=1}^{n-1}\frac{\ell_n}{\ell_j}\\
&\stackrel{\eqref{f522}}{\le} C''\a\max_{1\le j\le n-1}\t_j\cdot\sum_{j=1}^{\infty}2^{-j}
<C''\a\biggl[\sum_{j=1}^n\t_j^2\biggr]^{1/2}
\le\frac{c_0}{4T^s}\biggl[\sum_{j=1}^n\t_j^2\biggr]^{1/2},
\end{split}
\end{equation}
if $\a\le c_0[4C''T^s]^{-1},\ C''=C''(d,s)$. Clearly, $\a$ depends only on $d$ and $s$.

The inequality \eqref{f521} and the same arguments as in \eqref{f523} yield the estimate
\begin{equation}\label{f524}
|\tilde R_\nu^s(x')-\tilde R_{\nu}^s(x'')|<\frac{c_0}{4T^s}\biggl[\sum_{j=1}^n\t_j^2\biggr]^{1/2},
\quad x',x''\in E_{n,k}.
\end{equation}

Let $\tilde\EE=\tilde\EE(\nu_0)$ be the set defined by \eqref{f52} for the measure $\nu_0$. Fix some cube
$E_{n,k}\in\tilde\EE$, and set
$$
I_{n,k}(x):=\int_{E_{n,k}}\frac{y-x}{|y-x|^{s+1}}\,d\nu(x)=R_\nu^s(x)-\tilde R_{\nu}^s(x),\quad x\in E_{n,k}
$$
(we consider only the points $x$ for which $R_\nu^s(x)$ exists).
Let $E_{n,k^\ast}$ be the cube symmetric to $E_{n,k}$ with respect to the center $x_0$ of the cube $E_0$.
Choose $x\in E_{n,k}$, and let $x^\ast\in E_{n,k^\ast}$ be such that $x-x^\ast=x_{n,k}-x_{n,k^\ast}$. Clearly,
$I_{n,k}(x)=I_{n,k^\ast}(x^\ast)$. Moreover,
$$
\tilde R_{\nu_0}^s(x_{n,k})=-\tilde R_{\nu_0}^s(x_{n,k^\ast})
$$
by the symmetry of the measure $\nu_0$ with respect to $x_0$. Hence, for at least one couple of vectors
$I_{n,k}(x)$, $\tilde R_{\nu_0}^s(x_{n,k})$, or $I_{n,k^\ast}(x^\ast)$, $\tilde R_{\nu_0}^s(x_{n,k^\ast})$,
the angle between these vectors is less or equal to $\pi/2$. So, for the sets
\begin{align*}
\EE_{n,k}&:=\{x\in E_{n,k}:|I_{n,k}(x)+\tilde R_{\nu_0}^s(x_{n,k})|>|\tilde R_{\nu_0}^s(x_{n,k})|\}\\
\EE_{n,k^\ast}&:=\{x\in E_{n,k^\ast}:|I_{n,k^\ast}(x)+\tilde R_{\nu_0}^s(x_{n,k^\ast})|>
|\tilde R_{\nu_0}^s(x_{n,k^\ast})|\}
\end{align*}
we have
$$
\HH^d(\EE_{n,k})+\HH^d(\EE_{n,k^\ast})\ge \ell_n^d.
$$
For $x\in\EE_{n,k}$ we get
\begin{align*}
|R_{\nu}^s(x)|&=|I_{n,k}(x)+\tilde R_{\nu}^s(x)|\\
&=|I_{n,k}(x)+\tilde R_{\nu_0}^s(x_{n,k})+\tilde R_{\nu}^s(x_{n,k})-\tilde R_{\nu_0}^s(x_{n,k})
+\tilde R_{\nu}^s(x)-\tilde R_{\nu}^s(x_{n,k})|\\
&>|\tilde R_{\nu_0}^s(x_{n,k})|-|\tilde R_{\nu}^s(x_{n,k})-\tilde R_{\nu_0}^s(x_{n,k})|
-|\tilde R_{\nu}^s(x)-\tilde R_{\nu}^s(x_{n,k})|>\frac{c_0}{2T^s}\biggl[\sum_{j=1}^n\t_j^2\biggr]^{1/2}.
\end{align*}
Analogous estimates hold for $x^\ast\in\EE_{n,k^\ast}$. Hence,
$$
\HH^d\bigg\{x\in E_{n,k}\cup E_{n,k^\ast}:|R_{\nu}^s(x)|>\frac{c_0}{2T^s}\biggl[\sum_{j=1}^n\t_j^2\biggr]^{1/2}\bigg\}
\ge\HH^d(\EE_{n,k})+\HH^d(\EE_{n,k^\ast})\ge \ell_n^d
$$
for every $E_{n,k}\in\tilde\EE$. Since $\#\tilde\EE>c_12^{dn}$ (see \eqref{f53}), we get \eqref{f519} with
$c_3=c_0[2T^s]^{-1}$.
\end{proof}

The inequality \eqref{f26} is a particular case of Corollary \ref{co56}. Indeed, note
that $\mu=\rho\HH^d|_{E_n}$, where $\rho=2^{-dn}\ell_n^{-d}$. Multiplying both parts of \eqref{f519} by $\rho$,
we get \eqref{f26}.

\section{The capacity $\g_{s,+}$ of Cantor sets: proof of Theorem \ref{th23}}

As we mentioned in Section~2, our arguments are similar to those in \cite{MT} (which in turn use the ideas in \cite{MTV}). On the other hand, there are certain differences as well. For instance, in \cite{MT} the harmonicity of
the Riesz transform (outside the support) was used in an essential way. Since $s$ is not necessarily equal to $d-1$,
we cannot use harmonicity, and we cannot work with the measure supported on the boundaries of certain cubes, our measure will be supported also on the interior of these cubes. Moreover, and this is more essential, \cite{MT} uses a certain regularity of their Cantor sets. We have another type of Cantor sets, lacking this regularity (namely, not all our cubes $E_{j,l}$ are separated enough). This creates specific difficulties. There is a number of other differences (for example, we do not use Cotlar's inequality). Also, Mateu and Tolsa \cite{MT}, while claiming the result for all $s\in(0,d)$ (and for their Cantor sets), give the proof only for $s=d-1$. This is why we wish to present a full proof, even though it follows the idea of \cite{MT}.

We need the following characterization of $\g_{s,+}$ obtained in \cite{V}, Chapter~5:
\begin{equation}\label{f61}
\g_{s,+}(E)\approx\sup\{\|\eta\|:\eta\in \Sigma_s,\ \supp\eta\subset E,\
\pmb|\mathfrak{R}_{\eta}^s\pmb|\le1\},\quad 0<s<d
\end{equation}
(see \eqref{f21} for definition of $\Sigma_s$). Following \cite{MT}, we introduce the capacity
\begin{equation}\label{f62}
\g_{s,+}^\mu(E_n)=\sup\{\tau:0\le\tau\le1,\ \pmb|\mathfrak{R}_{\tau\mu}^s\pmb|\le1\},
\end{equation}
where $\mu$ is the probability measure defined in Section~1. Using \eqref{f25} and \eqref{f411}, we have
$$
\pmb|\mathfrak{R}_{\tau\mu}^s\pmb|=\tau\pmb|\mathfrak{R}_{\mu}^s\pmb|
\approx\tau\biggl[\sum_{j=1}^n\t_j^2\biggr]^{1/2}\approx\tau\biggl[\sum_{p=1}^m\T_p^2\biggr]^{1/2},\quad
\T_p=\frac{2^{-dj_p}}{\ell_{j_p}^s}\,.
$$
Hence,
\begin{equation}\label{f63}
\g_{s,+}^\mu(E_n)\approx\biggl[\sum_{j=1}^n\t_j^2\biggr]^{-1/2}\approx\biggl[\sum_{p=1}^m\T_p^2\biggr]^{-1/2},
\end{equation}
where the constants of comparison depend only on $d$ and $s$. It is easy to see that the measure
$\eta:=c\bigl[\sum_{j=1}^n\t_j^2\bigr]^{-1/2}\mu$ with $c=c(d,s)$ belongs to $\Sigma_s$. Now the relation \eqref{f61} and the upper bound in \eqref{f25} imply the estimate
$$
\g_{s,+}(E_n)>c\biggl[\sum_{j=1}^n\t_j^2\biggr]^{-1/2}\approx\g_{s,+}^\mu(E_n),\quad c=c(d,s).
$$
Thus, it is sufficient to prove that
\begin{equation}\label{f64}
\g_{s,+}(E_n)<C_0\g_{s,+}^\mu(E_n),\quad C_0=C_0(d,s).
\end{equation}

Let $\nu$ be a positive Radon measure supported on a compact set $E$ in $\R^d$, for which $\|R_{\nu}^s(x)\|_{L^\infty(\R^d)}\le1$. It is shown in \cite{MPV}, p.~217, that the last inequality implies the estimate
\begin{equation}\label{f65}
\nu(B(x,r))\le Cr^s,\quad x\in \R^d,\ r>0.
\end{equation}
The arguments in this part of the proof of Lemma~4.1 in \cite{MPV} are valid not only for $0<s<1$, but for $0<s<d$ as well
(the reference [P], Lemma~11 in \cite{MPV} should be replaced by [P], Lemma~3.1). For $s=d-1$, this fact is also noted in \cite{V}, p.~46. Hence, $\nu(E)\le C\HH^s(E)$. By the definition of $\g_{s,+}$ (see Section~2) we have
\begin{equation}\label{f66}
\g_{s,+}(E)\le C\HH^s(E),\ C=C(d,s),\text{ for any compact set $E\subset\R^d$}.
\end{equation}
(The inequality \eqref{f66} also follows from \cite[Lemma~3.2]{P}.)

We will prove \eqref{f64} by induction on $n$. The induction hypothesis is
\begin{equation}\label{f67}
\g_{s,+}(E_q)<C_0\g_{s,+}^\mu(E_q),\quad 1\le q<n,
\end{equation}
where the constant $C_0=C_0(d,s)$ will be specified later. Let
$$
S_k=\T_1^2+\T_2^2+\dots+\T_k^2,\quad 1\le k\le m.
$$
Suppose that $\frac12S_m<S_1$. Then by \eqref{f63} we get $\g_{s,+}^\mu(E_n)>c\ell_1^s$. This inequality together with \eqref{f66} yield the estimate
$$
\g_{s,+}(E_n)\le\g_{s,+}(E_1)\le C\ell_1^s<C'\g_{s,+}^\mu(E_n),\quad C'=C'(d,s).
$$
In particular, we get \eqref{f64} for $n=1$, if $C'<C_0$. Moreover, we can assume without loss of generality that $\frac12S_m\ge S_1$. Hence, there exists $K$, $1\le K<m$, such that
\begin{equation}\label{f68}
S_K\le\frac12S_m<S_{K+1}.
\end{equation}
We consider two cases.

{\bf Case 1.} For some constant $A_0=A_0(d,s)$ to be determined below,
$$
\g_{s,+}(E_{j_{K+1}-1,k}\cap E_n)> A_02^{-d(j_{K+1}-1)}\g_{s,+}(E_n),\quad k=1,\dots,2^{d(j_{K+1}-1)}.
$$
The set $E_{j_{K+1}-1,k}\cap E_n$ is constructed exactly in the same way as $E_n$, starting with $\s_{j_{K+1}}=\ell_{j_{K+1}}$ instead of $\s_1=\ell_1$, and with $q=n-j_{K+1}+1$ instead of $n$. Hence,
\begin{align*}
\g_{s,+}(E_n)&<A_0^{-1}2^{d(j_{K+1}-1)}\g_{s,+}(E_{j_{K+1}-1,k}\cap E_n)\stackrel{\eqref{f67}}{<}
A_0^{-1}2^{d(j_{K+1}-1)}C_0\g_{s,+}^{\mu}(E_{j_{K+1}-1,k}\cap E_n)\\
&\stackrel{\eqref{f63}}{<}A_0^{-1}C_0C2^{d(j_{K+1}-1)}\bigg[\sum_{p=K+1}^m\bigg(\frac
{2^{-d(j_p-j_{K+1}+1)}}{\ell_{j_p}^s}\bigg)^2\bigg]^{-1/2}
=A_0^{-1}C_0C\biggl[\sum_{p=K+1}^m\T_p^2\biggr]^{-1/2}\\
&=A_0^{-1}C_0C[S_m-S_K]^{-1/2}\stackrel{\eqref{f68}}{\le}\sqrt2A_0^{-1}C_0CS_m^{-1/2}\\
&\stackrel{\eqref{f63}}{<}A_0^{-1}C_0C'\g_{s,+}^{\mu}(E_n)<C_0\g_{s,+}^{\mu}(E_n),
\end{align*}
if $A_0\ge C'=C'(d,s)$. We get \eqref{f64}.

{\bf Case 2.} For the constant $A_0$ determined above,
\begin{equation}\label{f69}
\g_{s,+}(E_{j_{K+1}-1,k}\cap E_n)\le A_02^{-d(j_{K+1}-1)}\g_{s,+}(E_n),\quad k=1,\dots,2^{d(j_{K+1}-1)}.
\end{equation}
As in \cite{MT}, we again distinguish two cases, namely $\T_{K+1}^2>S_K$ and $\T_{K+1}^2\le S_K$. If $\T_{K+1}^2>S_K$,
then
$$
S_{K+1}=S_K+ \T_{K+1}^2<2\T_{K+1}^2,
$$
and we have
\begin{align*}
\g_{s,+}(E_n)&\le\g_{s,+}(E_{j_{K+1}})\stackrel{\eqref{f66}}{\le}C\HH^s(E_{j_{K+1}})\le C'2^{dj_{K+1}} \ell_{j_{K+1}}^s\\
&=\frac{C'}{\T_{K+1}}<\frac{\sqrt2C'}{S_{K+1}^{1/2}}\stackrel{\eqref{f68}}{<}\frac{2C'}{S_m^{1/2}}
\stackrel{\eqref{f63}}{<}C_0\g_{s,+}^{\mu}(E_n).
\end{align*}
Thus, \eqref{f64} holds if $C_0=C_0(d,s)$ is sufficiently big.

Suppose now that $\T_{K+1}^2\le S_K$. Then
\begin{equation}\label{f610}
\frac12S_m\stackrel{\eqref{f68}}{<}S_{K+1}\le2S_K<2S_m.
\end{equation}
We consider the measure
$$
\eta=\frac{\g_{s,+}(E_n)}{\HH^d(E_{j_K})}\HH^d\big|_{E_{j_K}}=
\frac{\g_{s,+}(E_n)}{2^{dj_K}\ell_{j_K}^d}\HH^d\big|_{E_{j_K}}.
$$
Clearly, $\|\eta\|=\g_{s,+}(E_n)$. We will show that $\pmb|\mathfrak{R}_{\eta}^s\pmb|\le C(d,s)$. Assuming this fact for a moment, we get
$$
\g_{s,+}(E_n)=\|\eta\|\stackrel{\eqref{f62}}{\le}C\g_{s,+}^\mu(E_{j_K})\stackrel{\eqref{f63}}{\le}
C'S_K^{-1/2}\stackrel{\eqref{f610}}{<}2C'S_m^{-1/2}\stackrel{\eqref{f63}}{<}C''\g_{s,+}^\mu(E_n),
$$
and \eqref{f64} follows.

To prove that $\pmb|\mathfrak{R}_{\eta}^s\pmb|$ is bounded, we will use the local $T(b)$ theorem of Christ \cite{C}. According to the Main Theorem~10 in \cite{C}, p.~605, it is enough to prove that $\eta$ satisfies the following conditions:

(i) $\eta(B(x,r))\le Cr^s,\ x\in E_{j_K},\ r>0$;

(ii) $\eta(B(x,2r))\le C\eta(B(x,r)),\ x\in E_{j_K},\ r>0$;

(iii) for each ball $B$ centered at a point in $E_{j_K}$, there exists a function $b_B$ in $L^\infty(\eta)$, supported on $B$, such that $|b_B|\le C$ and $|\mathfrak{R}_{\eta,\e}^s b_B|\le C$ $\eta$-almost everywhere on $E_{j_K}$, and
$\eta(B)\le C|\int b_B\,d\eta|$.

First we verify condition (i). Let $B=B(x,r)$ be a ball centered at $x\in E_{j_K}$. If $r<\ell_{j_K}$, we have
\begin{equation}\label{f611}
\begin{split}
\eta(B)\le C(d)\frac{\g_{s,+}(E_n)}{2^{dj_K}\ell_{j_K}^d}\,r^d\le C(d)\frac{\g_{s,+}(E_{j_K})}{2^{dj_K}\ell_{j_K}^d}\,r^d\\
\stackrel{\eqref{f66}}{\le}C'\frac{2^{dj_K}\ell_{j_K}^s}{2^{dj_K}\ell_{j_K}^d}\,r^d
=\frac{C'}{\ell_{j_K}^{d-s}}r^{d-s}r^s<C'r^s.
\end{split}
\end{equation}

Suppose that $r\ge\ell_{j_K}$. Let $j\le j_K$ be the least integer for which $\ell_j\le r$. Then $B$ may intersect at most $C=C(d)$ cubes $E_{j,k}$. Hence,
$$
\eta(B)\le C\eta(E_{j,k})=\frac{C\g_{s,+}(E_n)}{2^{dj}}
\le\frac{C\g_{s,+}(E_j)}{2^{dj}\ell_{j}^s}\,\ell_j^s\stackrel{\eqref{f66}}{\le}C'\ell_{j}^s\le C'r^s.
$$

Using the same kind of ideas, it is not difficult to verify condition (ii) as well.

Thus, we only need to check the hypothesis (iii). Again, let a ball $B=B(x',r)$ centered at $x'\in E_{j_K}$ be given. If $r\le2\sqrt{d}\ell_{j_K}$, we set $b_B=\chi_{B}$. Then
\begin{align*}
|\mathfrak{R}_{\eta,\e}^s b_B(x)|&\le\int_B\frac1{|x-y|^s}\, \frac{\g_{s,+}(E_n)}{2^{dj_K}\ell_{j_K}^d}\,d\HH^d(y)\\
&\le\frac{\g_{s,+}(E_{j_K})}{2^{dj_K}\ell_{j_K}^d}\int_0^{2\sqrt{d}\ell_{j_K}}\frac1{r^s}\,r^{d-1}\,dr
\stackrel{\eqref{f66}}{\le}\frac{C(d,s)}{\ell_{j_K}^{d-s}}\,\ell_{j_K}^{d-s}=C(d,s).
\end{align*}

Suppose that $r>2\sqrt{d}\ell_{j_K}$. Let $j_B$ be the least integer for which there is a cube of generation $j_B$ (i.e. the cube of $E_{j_B}$) contained in $\frac12B$. Clearly, $j_B\le j_K$ (since $B$ is centered at $x'\in E_{j_K}$). We will construct the function $b_B$ supported on $B$. By definition of the capacity $\g_{s,+}$, there is the positive Radon measure $\nu$ supported on $E_n$ such that
$\|R_{\nu}^s(x)\|_{L^\infty(\R^d)}\le1$ and $\nu(E_n)>\frac12\g_{s,+}(E_n)$. Hence, there is the cube $E_{j_B,k}$ for which
\begin{equation}\label{f612}
\nu(E_{j_B,k})\ge2^{-dj_B}\nu(E_n)>\frac12 2^{-dj_B}\g_{s,+}(E_n)=\frac12\eta(E_{j_B,k}).
\end{equation}

We need the localization lemma in \cite{MPV}. Let $Q=Q(w,\ell)$ be any cube (see \eqref{f15} for the notation $Q(w,\ell)$). Let $\f_Q$ be an infinitely differentiable function supported on $2Q$ and such that
$\|\partial^{\mathbf k}\f_Q\|_{L^\infty(\R^d)}\le C(s)\ell^{-|\mathbf k|}$, $0\le|\mathbf k|\le d$, $0\le\f_Q(y)\le1$, $\f_Q(y)=1$ on $Q$. By Lemma~3.1 in \cite{MPV}, p.~207,
\begin{equation}\label{f613}
\|R_{\f_Q\nu}^s\|_{L^\infty(\R^d)}\le C,\quad C=C(d,s).
\end{equation}

To simplify notation, set $\f^B=\f_Q$ for $Q=E_{j_B,k}$. At first we define $b_B$ when $E_{j_B,k}\subset\frac12B$. In this case, supp\,$\f^B\subset2E_{j_B,k}\subset B$. We set
\begin{equation}\label{f614}
b_B(x)=b_{j_B,k}(x)=\sum_{i:E_{{j_K},i}\subset 2E_{j_B,k}} \frac{(\f^B\nu)(E_{j_K,i})}{\eta(E_{j_K,i})}\,\chi_{E_{j_K,i}}(x).
\end{equation}
If $E_{j_B,k}\not\subset\frac12B$, we choose any cube $E_{j_B,i}\subset\frac12B$, and define $b_B$ by translation of $b_{j_B,k}$, namely
$$
b_B(x)=b_{j_B,k}(x+x_{j_B,k}-x_{j_B,i})
$$
(we recall that $x_{j,i}$ is the center of $E_{j,i}$). Clearly,
$$
\int b_B(x)\,d\eta\ge\nu(E_{j_B,k})\stackrel{\eqref{f612}}{>}\frac12\eta(E_{j_B,k})> c\,\eta(B)
$$
(the last inequality follows from the fact that there are at most $A=A(d)$ cubes of $j_B$-th generation in $B$).

To prove that $b_B$ is bounded, we apply \eqref{f613} to a cube $Q=Q_{K+1,l}=Q(w_{K+1,l},2T\ell_{j_{K+1}})$ (see Section~1 for notations). By \eqref{f16}, the cubes $Q_{K+1,l}$ are separated. Hence, $\f_Q\nu=\nu|Q =\nu|(E_{j_{K+1}-1,l}\cap E_n)$, and by \eqref{f613} we have $\|R_{\nu|(E_{j_{K+1}-1,l}\cap E_n)}^s\|_{L^\infty(\R^d)}\le C$. Thus,
\begin{align*}
\nu(E_{j_{K+1}-1,l}\cap E_n)&\le C\g_{s,+}(E_{j_{K+1}-1,l}\cap E_n)\\
&\stackrel{\eqref{f69}}{\le}CA_02^{-d(j_{K+1}-1)}\g_{s,+}(E_n)=CA_02^{-d(j_{K+1}-1)}\|\eta\|.
\end{align*}
Since $E_{{j_K},i}$ consists of $2^{d(j_{K+1}-j_K-1)}$ cubes $E_{j_{K+1}-1,l}$, we get the estimate
\begin{equation}\label{f615}
\begin{split}
\nu(E_{j_K,i})&\le CA_02^{d(j_{K+1}-j_K-1)}\,2^{-d(j_{K+1}-1)}\|\eta\|\\
&=C'2^{-dj_{K}}\|\eta\|=C'\eta(E_{j_K,i}),\quad i=1,\dots,2^{dj_K}.
\end{split}
\end{equation}
Now \eqref{f614} implies that
\begin{equation}\label{f616}
0\le b_B(x)\le C,\quad x\in\R^d,\quad C=C(d,s).
\end{equation}

To complete the proof we only need to check that
\begin{equation}\label{f617}
|\mathfrak{R}_{\eta,\e}^s b_B(x)|\le C\ \eta\text{-a.~e. on }E_{j_K}.
\end{equation}
By translation invariance, it is sufficient to consider the case $b_B(x)=b_{j_B,k}(x)$. The same estimates as in \eqref{f611} yield the inequality $\eta(B(x,t))\le C\ell_{j_K}^{s-d}t^d$, $t>0$. Integrating by parts, for every
$a\in(0,4\sqrt d\ell_{j_K})$ we get
\begin{multline*}
\bigg|\int_{a\le|x-y|<4\sqrt d\ell_{j_K}}\frac{x-y}{|x-y|^{s+1}}b_B(y)\,d\eta(y)\bigg|\stackrel{\eqref{f616}}{\le}
C\int_0^{4\sqrt d\ell_{j_K}}\frac1{t^s}\,d\eta(B(x,t))\\
=C\bigg(\frac{\eta(B(x,4\sqrt d\ell_{j_K}))}{(4\sqrt d\ell_{j_K})^s}
+s\int_0^{4\sqrt d\ell_{j_K}}\frac{\eta(B(x,t))}{t^{s+1}}\,dt\bigg)=C'(d,s).
\end{multline*}
Thus, it is enough to prove \eqref{f617} for $\e\ge4\sqrt d\ell_{j_K}$ and $r>2\sqrt d\ell_{j_K}$.

Set $\f^{(\e)}=\f_Q$ for $Q=Q(x,\e/\sqrt d)$. Then supp\,$\f^{(\e)}\subset B(x,\e)$. As before, let $\f^B=\f_Q$ for $Q=E_{j_B,k}$. Applying \eqref{f613} to the measure $\nu$ with $\f_Q=\f^B$, and then to the measure $\f^B\nu$ instead of $\nu$ and with $\f_Q=\f^{(\e)}$, we obtain
\begin{equation}\label{f618}
\|R_{\f^B(1-\f^{(\e)})\nu}^s\|_{L^\infty(\R^d)}\le
\|R_{\f^B\nu}^s\|_{L^\infty(\R^d)}+
\|R_{\f^{(\e)}\f^B\nu}^s\|_{L^\infty(\R^d)}<C,\quad C=C(d,s).
\end{equation}
Let
$$
\CC=\{\cup E_{j_K,i}:E_{j_K,i}\cap2E_{j_B,k}\ne\varnothing,\ E_{j_K,i}\subset\R^d\setminus B(x,\e)\}.
$$
Suppose that $y\in(\supp\f^B(1-\f^{(\e)})\nu)\setminus\CC$. Then $y\in E_{j_K,i}$ for which
$E_{j_K,i}\cap(B(x,\e)\setminus Q(x,\e/\sqrt d))\ne\varnothing$. Hence,
$$
\frac{\e}{4\sqrt d}<\frac{\e}{2\sqrt d}-\ell_{j_K}\le|y-x|\le\e+\sqrt d\ell_{j_K}<5\e.
$$
Therefore,
\begin{multline*}
\bigg|\int_{(\supp\f^B(1-\f^{(\e)})\nu)\setminus\CC}\frac{x-y}{|x-y|^{s+1}}\,d(\f^B(1-\f^{(\e)})\nu)(y)\bigg|\\
<\frac{C}{\e^s}\cdot(\f^B(1-\f^{(\e)})\nu)(B(x,5\e))\stackrel{\eqref{f65}}{<}C'(d,s).
\end{multline*}
This estimate and \eqref{f618} imply the inequality
\begin{equation}\label{f619}
\|R_{(\f^B(1-\f^{(\e)})\nu)|\CC}^s\|_{L^\infty(\R^d)}=
\|R_{\f^B\nu|\CC}^s\|_{L^\infty(\R^d)}\le C,\quad C=C(d,s).
\end{equation}
In the same way we will show that
\begin{equation}\label{f620}
\bigg|\int_{\{|y-x|\ge\e\}\setminus\CC}\frac{x-y}{|x-y|^{s+1}}\,b_B(y)\,d\eta(y)\bigg|<C.
\end{equation}
Indeed, $\f^{(\e)}(y)=0$ for $|y-x|\ge\e$, and hence
$$
(\supp b_B\setminus B(x,\e))\setminus\CC\subseteq(\supp\f^B(1-\f^{(\e)})\nu)\setminus\CC.
$$
The same arguments as above together with \eqref{f616} and the property (i) of $\eta$ imply \eqref{f620}.

It remains to establish the inequality
\begin{equation}\label{f621}
\bigg|\int_{\CC}\frac{x-y}{|x-y|^{s+1}}\,b_B(y)\,d\eta(y)
-\int_{\CC}\frac{x-y}{|x-y|^{s+1}}\,d(\f_B\nu)(y)\bigg|<C \text{ for $\eta$-a.e. }x\in\R^d.
\end{equation}
Then \eqref{f617} will follow from \eqref{f619} and \eqref{f620}.

For every cube $E_{j_K,i}\subset\CC$ we have
\begin{align*}
\bigg|&\int_{E_{j_K,i}}\frac{x-y}{|x-y|^{s+1}}\,b_B(y)\,d\eta(y)
-\int_{E_{j_K,i}}\frac{x-y}{|x-y|^{s+1}}\,d(\f_B\nu)(y)\bigg|\\
\stackrel{\eqref{f614}}{=}
\bigg|&\int_{E_{j_K,i}}\frac{x-y}{|x-y|^{s+1}}\,\frac{(\f^B\nu)(E_{j_K,i})}{\eta(E_{j_K,i})}\,d\eta(y)
-\int_{E_{j_K,i}}\frac{x-y}{|x-y|^{s+1}}\,d(\f_B\nu)(y)\bigg|\\
=\bigg|&\int_{E_{j_K,i}}\bigg(\frac{x-y}{|x-y|^{s+1}}-\frac{x-x_{j_K,i}}{|x-x_{j_K,i}|^{s+1}}\bigg)
\frac{(\f^B\nu)(E_{j_K,i})}{\eta(E_{j_K,i})}\,d\eta(y)\\
-&\int_{E_{j_K,i}}\bigg(\frac{x-y}{|x-y|^{s+1}}-\frac{x-x_{j_K,i}}{|x-x_{j_K,i}|^{s+1}}\bigg)\,
d(\f_B\nu)(y)\bigg|\stackrel{\eqref{f513}}{\le}C\ell_{j_K}\frac{(\f_B\nu)(E_{j_K,i})}{|x-x_{j_K,i}|^{s+1}}\,.
\end{align*}
Hence, the left-hand side of \eqref{f621} does not exceed
$$
C\ell_{j_K}\sum_{i:E_{j_K,i}\subset\CC}\frac{(\f_B\nu)(E_{j_K,i})}{|x-x_{j_K,i}|^{s+1}}
<C'\ell_{j_K}\int_{\e}^\infty\frac{\nu(B(x,t))}{t^{s+2}}\,dt\stackrel{\eqref{f65}}{<}
C''\ell_{j_K}\int_{\e}^\infty\frac{dt}{t^2}<C(d,s),
$$
since $\e>\ell_{j_K}$. Theorem \ref{th23} is proved.

\section{Applications}

We start with the extension of Theorem \ref{th23} to infinite sequences $\{\s_j\}$. Let $\s_1$, $\s_2,\dots$ be positive numbers such that
\begin{equation}\label{f71}
2\s_{j+1}\le\s_j,\quad j=1,2,\dots,
\end{equation}
and let numbers $\a\in(0,\frac12)$, $T\in(1,\frac1{2\a})$ be given. We define the set $J=\{j_1,j_2,\dots\}$ and the ``regularized'' sequence $\{\ell_j\}$ in the same way as in Section~1 for $j<n$. Namely, set $j_1=1$. If $j_p\in J$, then $j_{p+1}$ is the least $j>j_p$ for which $\s_j\le\a2^{-(j-j_p)}\s_{j_p}$. Possibly, $\s_j>\a2^{-(j-j_p)}\s_{j_p}$ for all $j>j_p$. In this case the sequence $J$ is finite: $J=\{j_1,\dots,j_m\}$, $m\ge1$. Clearly,
\begin{equation}\label{f72}
J \text{ is infinite if and only if }\lim_{j\to\infty}2^j\s_j=0.
\end{equation}
Furthermore, we set
$$
\ell_j=2^{-(j-j_p)}\s_{j_p},\quad j_p\le j<j_{p+1},\quad p=1,2,\dots
$$
If $J$ is finite, then $p=1,\dots,m$ with $j_{m+1}:=\infty$. The Cantor set $E$ is defined by the relation
$$
E=\bigcap_{j_p\in J}E_{j_p},
$$
where the sets $E_{j_p}$ are defined in Section~1. In particular, if $J$ is finite then $E=E_{j_m}$.

\begin{theorem}\label{th71} Let $d\ge1$, $s\in(0,d)$, and let a sequence $\{\s_j\}_1^\infty$ satisfies \eqref{f71}. Then
\begin{equation}\label{f73}
c\biggl[\sum_{j=1}^\infty\biggl(\frac{2^{-dj}}{\s_j^s}\biggr)^2\biggr]^{-1/2}\le\g_{s,+}(E)
\le C\biggl[\sum_{j=1}^\infty\biggl(\frac{2^{-dj}}{\s_j^s}\biggr)^2\biggr]^{-1/2},
\end{equation}
where the positive constants $c,C$ and the parameters $\a$, $T$ of the Cantor set $E$ depend only on $d$ and $s$.
\end{theorem}

\begin{proof} Suppose that $J$ is finite. Then
$$
\sum_{j=j_m+1}^\infty\biggl(\frac{2^{-dj}}{\s_j^s}\biggr)^2
<\sum_{j=j_m+1}^\infty\biggl(\frac{2^{-dj}}{\a^s2^{-(j-j_m)s}\s_{j_m}^s}\biggr)^2
=\frac1{\a^{2s}}\biggl(\frac{2^{-dj_m}}{\s_{j_m}^s}\biggr)^2\sum_{j=j_m+1}^\infty2^{-2(d-s)(j-j_m)}.
$$
Now \eqref{f73} follows from \eqref{f27} with $n=j_m$.

Assume that $J$ is infinite. Since
$$
\g_{s,+}(E)\le\g_{s,+}(E_{j_p})\stackrel{\eqref{f27}}{\le}
C\biggl[\sum_{j=1}^{j_p}\biggl(\frac{2^{-dj}}{\s_j^s}\biggr)^2\biggr]^{-1/2},\quad p=1,2,\dots,
$$
we get the estimate from above. We also get \eqref{f73} (that is $\g_{s,+}(E)=0$) if the series in \eqref{f73} diverges. Thus, we may assume that this series converges. The definition of $\g_{s,+}$ and Theorem~\ref{th23} imply the existence of measures $\nu_p$, $p=1,2,\dots$, such that
$$
\supp\nu_p=E_{j_p},\quad \|\nu_p\|\approx\biggl[\sum_{j=1}^{j_p}\biggl(\frac{2^{-dj}}{\s_j^s}\biggr)^2\biggr]^{-1/2},\quad
\|R_{\nu_p}^s\|_{L^\infty(\R^d)}\le1.
$$
We may extract a weakly convergent subsequence $\{\nu_{p_i}\}$. Denote by $\nu$ the weak limit of this subsequence as $i\to\infty$. Clearly,
$$
\supp\nu=E,\quad \|\nu\|\approx\biggl[\sum_{j=1}^{\infty}\biggl(\frac{2^{-dj}}{\s_j^s}\biggr)^2\biggr]^{-1/2}.
$$
For any $x\in\R^d\setminus\supp\nu_p$, we have $|R_{\nu_p}^s(x)|\le1$ (otherwise $\|R_{\nu_p}^s\|_{L^\infty(\R^d)}>1$ by continuity of $R_{\nu_p}^s(x)$ on $\R^d\setminus\supp\nu_p$). Hence,
$$
|R_{\nu}^s(x)|=\lim_{p\to\infty}|R_{\nu_p}^s(x)|\le1,\quad x\in\R^d\setminus\supp\nu.
$$
By \eqref{f72}, $\HH^d(E)=0$. Hence, $\|R_{\nu}^s\|_{L^\infty(\R^d)}\le1$, and Theorem \ref{th71} is proved.
\end{proof}

In \cite{ENV} we obtained estimates for the Hausdorff content of the set where the Riesz transform $R_\nu^s(x)$ is large. We also obtained certain relations between Hausdorff content and the capacity $\g_{s,+}$.
We are going to show that these estimates are attained on the Cantor sets defined above. The possibility of considering {\it arbitrary} sequences $\{\s_j\}$ satisfying \eqref{f71} enables us to prove the corresponding assertions for {\it any} gauge function $h$ (with the natural assumption that $\frac{h(r)}{r^d}$ is nonincreasing -- see the explanation below) without any additional conditions.

By a {\it gauge} (or {\it measure}) function, we shall understand any continuous strictly increasing function
$h:[0,+\infty)\to[0,+\infty)$ such that $h(0)=0$ and $\lim_{r\to+\infty}h(r)=+\infty$.

The {\it Hausdorff content $M_h(G)$} of a set $G\subset\R^d$ is defined by
$$
M_h(G)=\inf\sum_j h(r_j),
$$
where the infimum is taken over all (at most countable) coverings of $G$ by balls of radii $r_j$.
Later on we assume that $\frac{h(r)}{r^d}$ is nonincreasing. This condition, which may seem to be a regularity condition at the first glance, is actually
not a restriction at all. It was proved in \cite{AH}, p.~133, Proposition 5.18, that for any measure function $h$ either $M_h(G)=0$ for all $G\subset\R^d$,
or there is another measure function $h^\ast$ such that $\frac{h^\ast(r)}{r^d}$ is nonincreasing and for which Hausdorff contents $M_h$ and $M_{h^\ast}$ coincide up to a constant factor depending only on the dimension $d$.

For $P>0$, set
\begin{align*}
\ZZ(\nu,P)&=\{x\in \R^d: R_{\nu}^s(x) \text{ exists and }|R_{\nu}^s(x)|>P\},\\
\ZZ^\ast(\nu,P)&=\{x\in \R^d: R_{\nu,\ast}^s(x)>P\}.
\end{align*}
Clearly, $\ZZ(\nu,P)\subset\ZZ^\ast(\nu,P)$.
Let $h$ be a measure function, $N\ge2$, and let $h^{-1}$ be inverse to $h$. For a measure $\nu$ consisting of $N$ point charges, we have obtained the inequality
$$
\mathbf M\le C(s,d)\frac{\|\nu\|}{P}
\left[\int_{h^{-1}(0.1\,\mathbf M/N)}^{h^{-1}(\mathbf M)}\bigg(\frac{h(t)}{t^s}\bigg)^2\frac{dt}{t}\right]^{1/2},
$$
where $\mathbf M=M_h(\ZZ^\ast(\nu,P))$. This implicit estimate can be written in a simpler form using a certain auxiliary function (see \cite{ENV} for details).

\begin{proposition}\label{pr72} For every $\eta>0$ and $N\ge N_0(d,s)$, one can find a measure $\nu$ which is a linear combination of $N$ Dirac point masses, and such that $\|\nu\|=\eta$, and
\begin{equation}\label{f74}
\mathbf M>c(d,s)\frac{\|\nu\|}{P}
\left[\int_{h^{-1}(C\mathbf M/N)}^{h^{-1}(\mathbf M)}\bigg(\frac{h(t)}{t^s}\bigg)^2\frac{dt}{t}\right]^{1/2},\quad C=C(d,s)\ge1,
\end{equation}
for any $\mathbf M\ge M_h(\ZZ(\nu,P))$.
\end{proposition}

In \cite[Section~7]{ENV} a certain family of random sets and measures was introduced. It was proved that this family contains (random) measures $\nu$ with the properties indicated in Proposition \ref{pr72}. Thus, the existence of a measure with the desired properties was established, but the concrete measure was not presented. Below we give a non-probabilistic construction of such a measure.

\begin{proof}[Proof of Proposition \ref{pr72}] Without loss of generality we may assume that $N=2^{nd}$. It is sufficient to prove our assertion for $\eta=1$. For given $N$, $P$, we introduce the function
$$
\k(\s)=\frac1{Ph(\s)}\left[\int_{a_{\s}}^{\s}\bigg(\frac{h(t)}{t^s}\bigg)^2\frac{dt}{t}\right]^{1/2},\
\text{ where }\ h(a_\s)=N^{-1}h(\s).
$$
Clearly, $\k(\s)\to0$ as $\s\to\infty$ (we recall that $h(\infty)=\infty$, and $a_\s\to\infty$). Since $\frac{h(t)}{t^d}$ is nonincreasing,
\begin{equation}\label{f75}
\frac{h(t_2)}{h(t_1)}\le\bigg(\frac{t_2}{t_1}\bigg)^d,\quad 0<t_1\le t_2.
\end{equation}
In particular, $(\s/a_\s)^d\ge N$. Hence, $\k(\s)\to\infty$ as $\s\to0$. Thus, there exists $\s_0>0$ such that
\begin{equation}\label{f76}
\frac1{Ph(\s_0)}\left[\int_{a_{\s_0}}^{\s_0}\bigg(\frac{h(t)}{t^s}\bigg)^2\frac{dt}{t}\right]^{1/2}=C_4,\quad
\frac1{Ph(\s)}\left[\int_{a_{\s}}^{\s}\bigg(\frac{h(t)}{t^s}\bigg)^2\frac{dt}{t}\right]^{1/2}<C_4,\ \s>\s_0,
\end{equation}
where the constant $C_4>1$, depending only on $d$ and $s$, will be specified later.

Define $\s_j$ by the equalities
\begin{equation}\label{f77}
h(\s_j)=2^{-dj}h(\s_0),\quad j=1,\dots,n.
\end{equation}
By \eqref{f75}, $2\s_{j+1}\le\s_j,\ j=0,\dots,n-1$. Let $E_n$ be the Cantor set from Corollary \ref{co56}, and let $\nu$ be the probability measure consisting of $N=2^{nd}$ equal point masses $2^{-nd}$ located at the centers of the cubes $E_{n,k},\ k=1,\dots,2^{nd}$. We will prove that \eqref{f519} implies \eqref{f74}. Indeed,
\begin{equation}\label{f78}
\begin{split}
\sum_{j=1}^{n}\biggl(\frac{2^{-dj}}{\ell_j^s}\biggr)^2\stackrel{\eqref{f77}}{\approx}
&\frac1{h(\s_0)^2}\sum_{j=1}^{n}\biggl(\frac{h(\s_j)}{\s_j^s}\biggr)^2\approx
\frac1{h(\s_0)^2}\sum_{j=1}^{n}\int_{\s_{j}}^{\s_{j-1}}\bigg(\frac{h(t)}{t^s}\bigg)^2\frac{dt}{t}\\
=&\frac1{h(\s_0)^2}\int_{\s_n}^{\s_0}\bigg(\frac{h(t)}{t^s}\bigg)^2\frac{dt}{t}
\stackrel{\eqref{f76}}{>}\frac{P^2}{c_3^2}\,,
\end{split}
\end{equation}
if $C_4$ is big enough; here $c_3$ is the constant from \eqref{f519}, and $a_{\s_0}=\s_n$.

Define the measure $\psi_n$ by the equality
\begin{equation}\label{f79}
\psi_n=\frac{h(\s_0)}{\HH^d(E_n)}\,\HH^d|_{E_n}.
\end{equation}
Fix a ball $B(x,t)\subset\R^d$. Suppose that $t\le\ell_n$. Then
$$
\psi_n(B(x,t))\le C(d)\,\frac{h(\s_0)}{2^{dn}\ell_n^d}\,t^d
\stackrel{\eqref{f77}}{\le} C'\,\frac{h(\ell_n)}{\ell_n^d}\,t^d
\le C'\,\frac{h(t)}{t^d}\,t^d=C'h(t).
$$
Let $t\ge\ell_n$, and let $j$ be such that $\ell_j\le t<\ell_{j-1}$ (if $t\ge\ell_1$, the upper bound is absent). Then $B(x,t)$ intersects at most $A(d)$ cubes $E_{j,i}$ of $j$-th generation, that is at most $A(d)2^{d(n-j)}$ cubes $E_{n,k}$. Therefore
\begin{equation}\label{f710}
\psi_n(B(x,t))\le A(d)\,\frac{h(\s_0)}{\HH^d(E_n)}\,2^{d(n-j)}\ell_n^d=A(d)h(\s_0)2^{-dj}
\stackrel{\eqref{f77}}{\le}Ch(t).
\end{equation}
Hence, $M_h(G)\ge c(d,s)\psi_n(G)$ for every set $G$ in $\R^d$. We have
\begin{align*}
M_h(\ZZ(\nu,P))&\ge c\psi_n(\ZZ(\nu,P))\\
&\stackrel{\eqref{f78}}{\ge}
c\psi_n\bigg\{x\in E_n:|R_{\nu}^s(x)|> c_3\bigg[\sum_{j=1}^n\t_j^2\bigg]^{1/2}\bigg\}\stackrel{\eqref{f519}}{\ge}
ch(\s_0).
\end{align*}
Thus, for any $\mathbf M\ge M_h(\ZZ(\nu,P))$ we have $\mathbf M\ge ch(\s_0)$. If $\mathbf M\ge h(\s_0)$, then
$\s:=h^{-1}(\mathbf M)\ge \s_0$, and \eqref{f76} implies \eqref{f74} with $C\in[1,N)$. If
$ch(\s_0)\le\mathbf M<h(\s_0)$, then
\begin{align*}
\mathbf M&\ge ch(\s_0)\stackrel{\eqref{f76}}{=}\frac{c}{C_4P}
\left[\int_{h^{-1}(h(\s_0)/N)}^{\s_0}\bigg(\frac{h(t)}{t^s}\bigg)^2\frac{dt}{t}\right]^{1/2}\\
&\ge\frac{c}{C_4P}
\left[\int_{h^{-1}(c^{-1}\mathbf M/N)}^{h^{-1}(\mathbf M)}\bigg(\frac{h(t)}{t^s}\bigg)^2\frac{dt}{t}\right]^{1/2},
\end{align*}
and we get \eqref{f74} with $C=c^{-1}$ and $N_0\ge C$.
\end{proof}

\medskip

\noindent{\bf Remark.} One can see that the relation \eqref{f77} plays a crucial role in the proof of Proposition \ref{pr72} (as well as in the proof of the assertions below). Thus, additional assumptions on $\{\s_j\}$ imply certain conditions on $h$. For instance, the assumption $\t_{j+1}\le\t_j$ gives the unnatural restriction
$h(\s_{j+1})/h(\s_j)\le(\s_{j+1}/\s_j)^s$. For $h(t)=t^\beta$, this  means that $\beta\ge s$.

Moreover, if we assume that $(2+\d)\s_{j+1}\le\s_j$, $\d>0$, then 
$h(\s_{j+1})/h(\s_j)=2^{-d}\ge(\s_{j+1}/\s_j)^{d-\e}$, $\e>0$. Our assumption that $h(t)t^{-d}$ is nonincreasing, does not provide us with this property. Thus, we need additional conditions on $h$ (for instance, the stronger assumption that $h(t)t^{\e-d}$ is nonincreasing).

\medskip

Our next results concern the problem on the comparison of the capacity $\g_{s,+}$ and Hausdorff measure. It was proved in \cite{ENV} (see Theorem~10.1) that for each compact set $E\subset\R^d$,
$$
\g_{s,+}(E)\ge cM_h(E)\bigg[\int_0^{t_2}\bigg(\frac{h(t)}{t^{s}}\bigg)^2\frac{dt}t\bigg]^{-1/2},\quad 0<s<d,
$$
where $c$ depends only on $d$, $s$, and $t_2$ is defined by the equality $h(t_2)=M_h(E)$. This relation is sharp in the following sense.

\begin{proposition}\label{pr73} For any $s,d$ with $0<s<d$, and for any measure function $h$, there is a constant $C$, depending only on $d,s$, and a compact set $E$, such that $M_h(E)>0$, and
\begin{equation}\label{f711}
\g_{s,+}(E)\le CM_h(E)\bigg[\int_0^{t_2}\bigg(\frac{h(t)}{t^{s}}\bigg)^2\frac{dt}t\bigg]^{-1/2},\ \text{ where }\
h(t_2)=M_h(E).
\end{equation}
\end{proposition}

\begin{proof} Fix some $\s_0>0$, and define the infinite sequence $\{\s_j\}_{j=1}^\infty$ by the equalities \eqref{f77}. Let $E$ be the Cantor set from Theorem \ref{th71}. We claim that
\begin{equation}\label{f712}
ch(\s_0)\le M_h(E)\le Ch(\s_0).
\end{equation}
The upper bound is obvious. The lower bound for a finite set $J$ was proved above. If $J$ is infinite, we consider the weak limit $\psi$ of some weakly convergent subsequence of the sequence $\{\psi_{j_p}\}$ of the measures defined in \eqref{f79}. Clearly, $\psi(B(x,t))\le Ch(t)$ for any ball $B(x,t)$ (see \eqref{f710}). Hence, $M_h(E)\ge c\psi(E)=ch(\s_0)$.

Using the same estimates as in \eqref{f78}, we get
$$
\sum_{j=1}^{\infty}\biggl(\frac{2^{-dj}}{\ell_j^s}\biggr)^2\approx
\frac1{h(\s_0)^2}\int_{0}^{\s_0}\bigg(\frac{h(t)}{t^s}\bigg)^2\frac{dt}{t}
\stackrel{\eqref{f712}}{\approx}\frac1{h(t_2)^2}\int_{0}^{t_2}\bigg(\frac{h(t)}{t^s}\bigg)^2\frac{dt}{t}.
$$
This relation together with \eqref{f73} imply \eqref{f711}.
\end{proof}

For $h(t)=t^\beta$ we get the following assertion.

\begin{corollary}\label{cor74} Let $0<s<d,\ h(t)=t^\beta,\ \beta>s$. There is a compact set $E\subset\R^d$ such that
$$
\g_{s,+}(E)\le C\,(\beta-s)^{1/2}[M_h(E)]^{s/\beta},\quad C=C(d,s).
$$
\end{corollary}
This statement is a supplement to Corollary 10.2 in \cite{ENV}: for each compact set $E\subset\R^d$,
$$
\g_{s,+}(E)\ge c\,(\beta-s)^{1/2}[M_h(E)]^{s/\beta},\quad\text{where}\quad 0<s<d,\quad h(t)=t^\beta,\quad\beta>s,
$$
and $c$ depends only on $d$ and $s$.

We conclude with one more direct consequence of Proposition \ref{pr73}.

\begin{corollary}\label{cor75} Suppose that $\int_{0}\big(\frac{h(t)}{t^s}\big)^2\frac{dt}{t}=\infty$.
Then there exists a compact set $E$ such that $M_h(E)>0$, but $\g_{s,+}(E)=0$.
\end{corollary}
This statement demonstrates the sharpness of the following implication (Corollary 10.2 in \cite{ENV}): if $M_h(E)>0$ for a measure function $h$ with  $\int_{0}\big(\frac{h(t)}{t^s}\big)^2\frac{dt}{t}<\infty$, then $\g_{s,+}(E)>0$.

\medskip

\noindent{\bf Remark.} It is interesting to compare the condition $\int_{0}\big(\frac{h(t)}{t^s}\big)^2\frac{dt}{t}<\infty$ with the corresponding condition $\int_0\frac{h(t)}{t^{s}}\,\frac{dt}t<\infty$, arising in the analogous problem of the comparison of Hausdorff measure and the classical Riesz capacity $C_s(E)$. The latter is generated by potentials with the {\it positive} kernel $|x|^{-s}$ (see \cite{Ca}, \cite{E1}, Sections 1,\,2 and the references therein, and \cite{AH}, p.~147 for a more general setting). Clearly, $C_s(E)\le\g_{s,+}(E)$. The exponent 2 in $\int_{0}\big(\frac{h(t)}{t^s}\big)^2\frac{dt}{t}$ reflects the difference between potentials with the signed kernel $x|x|^{-s-1}$, and potentials with the positive kernel $|x|^{-s}$.

\medskip

The results obtained generalize and refine the corresponding statements in \cite{E1}, where the case $d=2$, $s=1$ was considered.

\medskip

\noindent{\bf Acknowledgements.} We are grateful to F. Nazarov whose idea helped to simplify greatly the proof of Theorem 2.1.

\end{document}